\documentclass[a4paper]{article}

\usepackage{amsmath,amssymb,amsbsy}
  \usepackage{paralist}
  \usepackage{graphics} 
  \usepackage{epsfig} 
\usepackage{amsthm}
\usepackage{tikz}
\usepackage[all]{xy}
 \usepackage{verbatim}

\newtheorem{theorem}{Theorem}[section]
\newtheorem{proposition}[theorem]{Proposition}
\newtheorem{corollary}[theorem]{Corollary}
\newtheorem{lemma}[theorem]{Lemma}
\newtheorem{main lemma}[theorem]{Main Lemma}

{\theoremstyle{definition}
\newtheorem{definition}{Definition}

\newtheorem{remark}{Remark}

\newcommand{\ddfrac}[2] {\frac{\displaystyle #1 }{\displaystyle #2} }

\newcommand{\R}{\mathbb{R}}

\newcommand{\N}{\mathbb{N}}

\def\beq{\begin{equation}}
\def\eeq{\end{equation}}
\def\pa{\partial}

\def\d{\delta}

\def\wt{\widetilde}
\def\wh{\widehat}
\def\f{\varphi}
\def\a{\alpha}

\def\eps{\varepsilon}

\DeclareMathOperator{\Ker}{Ker}
\DeclareMathOperator{\im}{Range}
\DeclareMathOperator{\loc}{loc}
\title{Bounded solutions for a forced bounded oscillator without friction\footnote{
Work partially supported by the PRIN2009 grant ``Critical Point Theory and
Perturbative Methods for Nonlinear Differential Equations''.\break
2010 \emph{AMS Subject Classification.} Primary: 34B15; secondary: 34C11, 49J35 \hfill\break
\emph{Key words.} Landesman-Lazer conditions, critical point theory, subharmonic solutions to periodic ODEs, Ambrosetti-Prodi problems.
}
}

\author{Nicola Soave and Gianmaria Verzini}

\date{June 4, 2013}

\begin{document}

\maketitle

\begin{abstract}
Under the validity of a Landesman-Lazer type condition, we prove the existence of solutions bounded on the real line,
together with their first derivatives, for some second order nonlinear differential equation of the form $\ddot u + g(u) = p(t)$, where the reaction term $g$ is bounded. The proof is variational, and relies on a dual version of the Nehari method for the existence of oscillating solutions to superlinear equations.
\end{abstract}

\section{Introduction}\label{sec:intro}

This paper concerns the existence of solutions, bounded on the real line together with their first derivative, for the differential equation
\beq\label{equation}
\ddot{u}+g(u)=p(t),
\eeq
where $g\in\mathcal{C}^2(\R)$ is bounded, increasing, and has exactly one inflection point,
and $p\in \mathcal{C}(\R)\cap L^\infty(\R)$ admits asymptotic average $A(p)\in\R$, that is
\[
\lim_{T\to+\infty} \frac{1}{T} \int_t^{t+T} p(s)\,ds = A(p),
\]
uniformly in $t\in\R$. Such an equation describes the forced motions of an oscillator
exhibiting saturation effects. As a model problem, the reader may think to the equation
\[
\ddot{u}+\arctan u=p(t),
\]
even though we do not require any symmetry assumption on the reaction term $g$.
Under the above assumption, the main result we prove is the following theorem.
\begin{theorem}\label{thm: intro}
Equation \eqref{equation} admits a bounded solution if and only if
\begin{equation}\label{eq:land_laz}
g(-\infty)<A(p)<g(+\infty).
\end{equation}
In such a case, equation \eqref{equation} admits a countable set of bounded solutions, having
arbitrarily large $L^\infty$-norm.
\end{theorem}
The motivation for our investigation relies on the papers \cite{Ah,Or}, which in turn have been inspired by some classical results of Landesman-Lazer type holding in the periodic framework. Such studies concern the equation
\beq\label{equation_with_friction}
\ddot{u}+c\dot u+g(u)=p(t),
\eeq
where $c\in\R$ and the continuous function $g$, not necessarily monotone, admits limits at $\pm\infty$, with the property that
\[
g(-\infty)<g(s)<g(+\infty)
\]
for every $s$. Also the cases $g(\pm\infty)=\pm\infty$ can be considered, requiring $g$ to be sublinear at infinity if $c=0$. When $p$ is $T$-periodic, it is nowadays well known that equation \eqref{equation} admits a periodic solution if and only if the Landesman-Lazer condition
\[
g(-\infty)<\frac{1}{T} \int_0^{T} p(s)\,ds<g(+\infty)
\]
is satisfied, regardless of the constant $c$; this result was first proved by Lazer, using the Schauder fixed point theorem, see \cite{La}. When $p$ is merely bounded, one would like to find analogous conditions for the search of bounded solutions. This problem was first studied by Ahmad \cite{Ah}, under the assumption that $p$
has asymptotic average, in the sense explained above; by means of techniques of the qualitative theory of dissipative equations, the existence of a bounded solution is characterized, whenever $c\neq0$, by \eqref{eq:land_laz}. The case in which $p$ is an arbitrary continuous function was solved by Ortega \cite{Or}, who assumes $c\neq0$ and  provides a sharp necessary and sufficient condition: \eqref{equation_with_friction} has a bounded solution if and only if $p$ can be written as $p^*+p^{**}$, where $p^*$ has bounded primitive and $p^{**}$ assumes values strictly contained between $g(-\infty)$ and $g(+\infty)$. This result relies on the Krasnoselskii's method of guiding functions,
and was generalized by Ortega and Tineo \cite{OrTi} to equations of higher order, using the notions of lower and upper averages of $p$; again, the condition $c\neq0$ sticks as a crucial assumption. Later, by means of the method of lower and upper solutions, Mawhin and Ward \cite{MaWa} achieved some results in the case $c=0$, but in the complementary situation in which $g(-\infty)\geq g(+\infty)$. Up to our knowledge, this last is the unique extension of the Landesman-Lazer theory to second order equations without friction, and the question in the case $g(-\infty)<g(+\infty)$ is still open. Under this perspective, in this paper we go back to the setting originally considered by Ahmad, and we prove that its aforementioned result holds also in the case $c=0$, at least for the particular class of $g$ that we consider.

The proof of our result is variational: we use a dual Nehari method which was first introduced in \cite{OrVe} to obtain bounded solutions in the case of a sublinear reaction (i.e. $g(s)=s^{1/3}$). The method consists in two steps.

Firstly, we consider the boundary value problem
\beq\label{positive BVP}
\begin{cases}
\ddot{u} +g(u )=p(t) & t \in (a,b),\\
u(a)=0=u(b),\\
u(t)>0 & t \in (a,b),
\end{cases}
\end{equation}
searching for solutions as minimizers of the action functional
\[
J_{(a,b)}(u):= \int_a^b \left[ \frac{1}{2}\dot{u}^2(t)- G(u(t)) + p(t)u(t)\right]\,dt
\]
in the weakly closed set $\{u \in H_0^1(a,b): \ u \ge 0\}$. In Section \ref{sec: existence minimi} we obtain some general properties of the nonnegative minimizers of $J_{(a,b)}$ in any interval $(a,b)$; in Section \ref{sec: BVP large} we prove that, when $b-a$ is sufficiently large, the minimizer $u_+(\cdot\,;a,b)$ is unique and solves problem \eqref{positive BVP}. The proof of these results is substantially different from the corresponding one in the sublinear case \cite{OrVe}: indeed in the present situation the nonlinearity $g$ and the forcing term $p$ have the same order of growth (they are both bounded), while, as far as $b-a$ is sufficiently large, the forced sub-linear problem can be considered as a small perturbation of the unforced one. This fact introduces a lot of complications, which we can overcome thanks to a careful analysis of the balance between $g$ and $p$, via measure theory tools, and of the asymptotic properties of the functional $J_{(a,b)}$ as $b-a \to +\infty$. Of course, analogous results can be obtained for negative minimizers $u_-(\cdot\,;a,b)$. To proceed, it is necessary to prove that $u_\pm(\cdot\,;a,b)$ is non-degenerate and that $J_{(a,b)}(u_\pm(\cdot\,;a,b))$ is differentiable as a function of $(a,b)$. This is the object of Sections \ref{sec: non deg}, \ref{sec: diff}, and it is the only part which requires $g\in\mathcal{C}^2$. We believe that this assumption can be weakened by a suitable approximating procedure, but we prefer to avoid further technicalities at this point.

Once the existence of one-signed solutions is established, in Section \ref{sec: sign-change} we juxtapose positive and negative minimizers with alternate signs to obtain oscillating solutions. Indeed, let us fix $k \ge 1$, a bounded interval $[A,B]$ sufficiently large, and let us consider the class of partitions
\[
\mathcal{B}_k:= \left\{ (t_1,\dots,t_k) \in \R^k \left| \begin{array}{l} A=: t_0 \le t_1 \le \dots \le t_k \le t_{k+1}:=B,\\
t_{i+1}-t_i \text{ is sufficiently large for any $i$} \end{array} \right.\right\}.
\]
For each partition $P=(t_1,\dots,t_k)$ of $\mathcal{B}_k$ there is a function $u_P$ obtained by juxtaposing $u_{\pm}(\cdot\,;t_i,t_{i+1})$ with alternate signs $+$ and $-$. In general, this function is not a solution of equation \eqref{equation}, because the derivatives $\dot{u}_P(t_i^\pm)$ may not coincide. We prove that these corner points disappear for the partition maximizing the quantity
\[
\psi(P) = \sum_{i=0}^k J_{(t_i,t_{i+1})} (u_\pm(\cdot\,;t_i,t_{i+1})).
\]
This argument provides a solution of \eqref{equation} having $k$ zeros in $[A,B]$, together with some estimates which depend only on the ratio $(B-A)/k$. Therefore, taking $A \to -\infty$, $B \to +\infty$ and $k \to +\infty$ in an appropriate way, one can pass to the limit and obtain the desired bounded solution.
In doing this, one must again modify the corresponding arguments in the sub-linear case, indeed they do not allow to treat the non-symmetric case $g(+\infty)-A(p)\neq A(p)-g(-\infty)$.

Incidentally, assuming $p$ to be $T$-periodic, a simple variation of the argument above allows to obtain the existence of infinitely many subharmonic solutions, i.e. solutions which have minimal period $nT$, $n\in\N$ (see Theorem \ref{thm: esistenza sub} at the end of the paper).

To conclude, we remark that also the case of infinite limits $g(\pm\infty)$ can be treated by variational methods. On one hand, as already mentioned, infinitely many bounded solutions for equation \eqref{equation} were obtained in \cite{OrVe} when $g(s)=|s|^{q-1}s$, $0<q<1$, and $p\in L^{\infty}(\R)$. On the other hand, the original Nehari method, together with a limiting procedure, allows to obtain an analogous result also when $g$ is superlinear at infinity, as done in \cite{TeVe,Ve}.

\section{Preliminaries}

It is not difficult to check that if equation \eqref{equation} admits a bounded solution with bounded derivative, then necessarily condition \eqref{eq:land_laz} is satisfied. Indeed, by integrating equation \eqref{equation} in $(t,t+T)$, we obtain
\[
\frac{\dot{u}(t+T)-\dot{u}(t)}{T} = \frac{1}{T}\int_t^{t+T}  \left( p(s)- g(u(s)) \right) \, ds.
\]
Since $\dot{u}$ is bounded, passing to the limit as $T \to +\infty$ we deduce that the left hand side tends to $0$, so that
\begin{equation}\label{necessity, eq 1}
\begin{split}
0 & =\lim_{T \to +\infty} \frac{1}{T}\int_t^{t+T}  \left( p(s)- g(u(s)) \right) \, ds \\
&= A(p) - \lim_{T \to +\infty} \frac{1}{T} \int_t^{t+T} g(u(s))\,ds.
\end{split}
\end{equation}
Now, the boundedness of $u$ and the monotonicity of $g$ implies also that for every $s \in \R$
\begin{equation}\label{necessity, eq 2}
g(-\infty) < g\left(-\|u\|_\infty\right) \le g(u(s)) \le g\left(\|u\|_\infty\right) < g(+\infty),
\end{equation}
and a comparison between \eqref{necessity, eq 1} and \eqref{necessity, eq 2} gives the desired result
(in fact, from this point of view, it is sufficient that $g(-\infty)<g(s)<g(+\infty)$ for every $s$).

We observe that, by means of suitable translations, it is not restrictive to assume that
\begin{equation}\tag{h1}\label{h1}
\begin{split}
&\text{$g(0)=0$, $g \in \mathcal{C}^2(\R)$ is bounded, strictly increasing in $\R$,}\\
&\text{strictly concave in $(0,+\infty)$ and strictly convex in $(-\infty,0)$}.
\end{split}
\end{equation}
We denote as $G$ the primitive of $g$ vanishing in $0$, and
\[
\lim_{s \to \pm \infty} g(s) = g_\pm,
\]
so that
\[
\lim_{s \to \pm \infty} \frac{G(s)}{s}=g_\pm \quad \text{and} \quad g_- < \frac{G(s)}{s} < g_+ \quad \forall s \in \R.
\]
As far as the function $p$ is concerned, as we already mentioned, we assume that $p \in \mathcal{C}(\R) \cap L^\infty(\R)$ is such that for every $\eps>0$ there exists $\bar T>0$ such that if $T> \bar T$ then
\[
\sup_{t \in \R} \left| \frac{1}{T} \int_t^{t+T} p(s)\,ds - A(p) \right| < \eps,
\]
in such a way that
\begin{equation}\tag{h2}\label{h2}
\begin{split}
&\text{$p$ is bounded and continuous in $\R$}, \\
&\text{and has asymptotic average $g_- < A(p) < g_+$}.
\end{split}
\end{equation}
Note that we do not make any assumption on the $L^\infty$ norm of $p$.

In view of the previous considerations and notations, we can rephrase Theorem \ref{thm: intro} as follows.
\begin{theorem}\label{thm: main thm}
Under assumptions \eqref{h1}-\eqref{h2}, there exists a sequence $(u_m)$ of solutions of \eqref{equation} defined in $\R$, with $u_m,\dot u_m \in L^\infty(\R)$ and $\|u_m\|_\infty \to \infty$ as $m \to \infty$. Moreover, each $u_m$ has infinitely many zeros in $\R$.
\end{theorem}
\section{Existence and basic properties of nonnegative minimizers}\label{sec: existence minimi}
In this section we deal with the boundary value problem \eqref{positive BVP}:
\[
\begin{cases}
\ddot{u}(t)+g(u(t))=p(t) & t \in (a,b),\\
u(a)=0=u(b),\\
u(t)>0 & t \in (a,b).
\end{cases}
\]
We seek solutions as minimizers of the related action functional
\[
J_{(a,b)}(u):= \int_a^b \left[ \frac{1}{2}\dot{u}^2(t)- G(u(t)) + p(t)u(t)\right]\,dt
\]
in the $H^1$-weakly closed set
\[
H_0^1(a,b)^+:=\{u \in H_0^1(a,b): \ u \ge 0\}.
\]
We introduce the value
\[
\f^+(a,b):= \inf_{u \in H_0^1(a,b)^+} J_{(a,b)}(u).
\]
\begin{remark}\label{rmk: sul pb negativo}
Of course, even though in the following we focus on positive solutions, negative ones can be treated similarly as well, seeking solutions to the boundary value problem
\begin{equation*}
\begin{cases}
\ddot{u}(t) +g(u(t)) = p(t)  & t \in (a,b) \\
u(a)=0=u(b) \\
u(t)<0 & t \in (a,b)
\end{cases}
\end{equation*}
associated to the candidate critical value
\[
\varphi^-(a,b):= \inf_{u \in H_0^1(a,b)^-} J_{(a,b),p}(u),
\]
where $H_0^1(a,b)^-:=\{ u \in H_0^1(a,b): u \le 0\}$. Indeed, the two problems are related by the change of variable $v= -u$, $\bar g(s) = -g(-s)$ and $\bar p = -p$, and $\bar g$, $\bar p$ satisfy \eqref{h1}-\eqref{h2} if and only if $g$, $p$ do. In particular, when dealing with negative solutions, in all the explicit constants we will find the quantity $g_\pm$ should be replaced by $-g_\mp$, and $A(p)$ by $-A(p)$.
\end{remark}
\begin{lemma}
The value $\f^+(a,b)$ is a real number and it is achieved by $u_{(a,b)} \in H_0^1(a,b)^+$.
\end{lemma}
\begin{proof}
It is not difficult to check that $J_{(a,b)}$ is weakly lower semi-continuous and coercive, so that the direct method of the calculus of variations applies.
\end{proof}
In what follows we are going to show, that, if $(a,b)$ is sufficiently large, a minimizer $u_{(a,b)}$ is an actual solution of \eqref{positive BVP}; this is not obvious, because in principle $u_{(a,b)}$ could vanish somewhere. Having in mind to let $(a,b)$ vary and wishing to catch the behaviour of the minimizers $u_{(a,b)}$ under variations of the domain, it is convenient to introduce suitable scaling to work on a common time-interval. To be precise, for every $u \in H_0^1(a,b)^+$ we can define
\beq\label{scaling}
\wh{u}(t):= \frac{1}{(b-a)^2}u(a+t(b-a)) \quad \Longleftrightarrow \quad u(t)= (b-a)^2\wh{u}\left(\frac{t-a}{b-a}\right),
\eeq
and $\wh{p}_{(a,b)}(t):= p(a+t(b-a))$.
Of course, $\wh{u} \in H_0^1(0,1)^+$ and
\begin{multline}\label{scaled functional}
J_{(a,b)}(u) =
(b-a)^3 \int_0^1 \left[ \frac{1}{2}\dot{\wh{u}}\, ^2(t)-\frac{1}{(b-a)^2}G((b-a)^2\wh{u}(t))+ \wh{p}_{(a,b)}(t) \wh{u}(t) \right]\,dt \\=: (b-a)^3\wh{J}_{(a,b)}(\wh{u}).
\end{multline}
This reveals that the minimizations of $J_{(a,b)}$ in $H_0^1(a,b)^+$ and of $\wh{J}_{(a,b)}$ in $H_0^1(0,1)^+$ are equivalent; in particular, the function $\wh{u}_{(a,b)}$ defined by \eqref{scaling} with $u=u_{(a,b)}$ is a minimizer of $\wh{J}_{(a,b)}$ in $H_0^1(0,1)^+$. \\
The Euler-Lagrange equation associated to the functional $\wh{J}_{(a,b)}$ yields to the research of solutions to
\beq\label{scaled problem}
\begin{cases}
\ddot{w}(t)+g((b-a)^2 w(t)) =\wh{p}_{(a,b)}(t) & \text{in $(0,1)$}\\
w(0)=0=w(1) \\
w(t)>0 & \text{in $(0,1)$}.
\end{cases}
\eeq
Our aim is to show that if $b-a$ is sufficiently large than a minimizer $\wh{u}_{(a,b)}$ is an actual solution of \eqref{scaled problem}. We start showing that where it is positive it solves equation \eqref{equation}, and it is of class $\mathcal{C}^1$ in the whole $(0,1)$.
\begin{lemma}\label{smoothness minimizers}
Let $(c,d) \subset (0,1)$ be such that
\[
\wh{u}_{(a,b)}>0 \qquad \text{in $(c,d)$}.
\]
Then $\wh{u}_{(a,b)}$ is a classical solution of the first equation in \eqref{scaled problem} in $(c,d)$. Moreover, if $c>0$ then $\dot{\wh{u}}_{(a,b)}(c^+)=0$, and if $d<1$ then $\dot{\wh{u}}_{(a,b)}(d^-)=0$.
\end{lemma}
\begin{proof}
The fact that $\wh{u}_{(a,b)}$ is a (classical) solution in $(c,d)$ follows from the extremality of $\wh{u}_{(a,b)}$ with respect to variations with compact support in $(c,d)$.\\
Now we assume that $c>0$ and prove that $\dot{\wh{u}}_{(a,b)}(c^+)=0$. By contradiction, let $\dot{\wh{u}}_{(a,b)}(c^+)=\xi>0$. Given $\varepsilon>0$ small enough such that $[c-\eps,c+\eps] \subset (0,d)$, we consider the set
\[
\Lambda_\eps:=\left\{ v \in H^1(c-\eps,c+\eps): v(c\pm \eps)=\wh{u}_{(a,b)}(c\pm \eps)\right\}.
\]
As
\[
\|v\|_{\infty} \le \sqrt{2 \eps} \|\dot{v}\|_2 + \min\{\wh{u}_{(a,b)}(c-\eps), \wh{u}_{(a,b)}(c+\eps) \} \qquad \forall v \in \Lambda_\eps,
\]
the functional $\wh{J}_{(a,b)}$ (considered on the interval $(c-\eps,c+\eps)$) is bounded below and coercive in the weakly closed set $\Lambda_\eps$, so that there exists a minimizer $v_\eps$. Clearly, $v_\eps \in \mathcal{C}^2(c-\eps,c+\eps)$ and is a solution of
\beq\label{eq v_eps}
\ddot{v}_\eps(t)+g((b-a)^2v_\eps(t))=\wh{p}_{(a,b)}(t).
\eeq
Since the restriction $\wh{u}_{(a,b)}$ is not differentibale in $c$, we deduce
\[
\wh{J}_{(a,b)}(v_\eps) < \wh{J}_{(a,b)}(\wh{u}_{(a,b)}|_{(c-\eps,c+\eps)}).
\]
We claim that $v_\eps \ge 0$ in $(c-\eps,c+\eps)$. If $v_\eps$ is monotone, this follows from its boundary conditions. If it is not monotone, there exists $\tau \in (c-\eps,c+\eps)$ such that $\dot{v}_\eps(\tau)=0$. As a consequence, from equation \eqref{eq v_eps} it follows that
\[
\|\dot{v}\|_\infty \le \left( \|g\|_{\infty} + \|p\|_{\infty}\right)2\eps,
\]
and hence, for every $t \in (c-\eps,c+\eps)$, we have
\[
v_\eps(t) \ge v_\eps(c+\eps)- |v_\eps(c+\eps)-v_\eps(t)| \ge \wh{u}_{(a,b)}(c+\eps)-\left(\|g\|_{\infty}+ \|p\|_{\infty}\right)4\eps^2.
\]
Now, $\wh{u}_{(a,b)}(c+\eps)= \xi \eps + O(\eps^2)$, so that at least for $\eps$ small enough we have $v_\eps(t) \ge 0$ in $(c-\eps,c+\eps)$, as announced. This implies that the function
\[
\wt u(t):=\begin{cases}
\wh{u}_{(a,b)}(t) & t \in [0,c-\eps) \cup (c+\eps,1], \\
v_\eps(t) & t \in (c-\eps,c+\eps)
\end{cases}
\]
stays in $H_0^1(0,1)^+$ and, clearly, $\wh{J}_{(a,b)}(\wt u) < \wh{J}_{(a,b)}(\wh{u}_{(a,b)})$, in contradiction with the minimality of $\wh{u}_{(a,b)}$.
\end{proof}
In the following lemma we prove that the family  of the minimizers $\{\wh{u}_{(a,b)}\}$ is uniformly bounded and equi-Lipschitz-continuous.
\begin{lemma}\label{bound norm minimizers}
For every $(a,b) \subset \R$ and any $\wh{u}_{(a,b)}$, it holds
\[
\begin{split}
| \wh{u}_{(a,b)}(t) | & \le ( \|g\|_{\infty}+\| p\|_{\infty}) \qquad \forall t \in (0,1)\\
| \dot{\wh{u}}_{(a,b)} (t)| & \le ( \|g\|_{\infty}+\| p\|_{\infty}) \qquad \forall t \in (0,1).
\end{split}
\]
\end{lemma}
\begin{proof}
Let $(c,d) \subset [0,1]$ be such that $\wh{u}_{(a,b)}>0$ in $(c,d)$, vanishing at $c$ and $d$. From Lemma \ref{smoothness minimizers} it follows that
\[
|\ddot{\wh{u}}_{(a,b)}(t)| \le \left| g((b-a)^2 \wh{u}_{(a,b)}(t)) \right| + | p (t)| \le \|g\|_{\infty}+\|p\|_{\infty} \qquad \forall t \in (c,d).
\]
Since $\wh{u}_{(a,b)}(c)=0=\wh{u}_{(a,b)}(d)$ and $\wh{u}_{(a,b)} \in \mathcal{C}^1(0,1)$, there exists $\tau \in (c,d)$ such that $\dot{\wh{u}}_{(a,b)}(\tau)=0$. Hence
\[
|\dot{\wh{u}}_{(a,b)}(t)| \le |\dot{\wh{u}}_{(a,b)}(\tau)| +  \|g\|_{\infty}+ \|p\|_{\infty}= \|g\|_{\infty}+ \|p\|_{\infty}  \qquad \forall t \in (c,d).
\]
Since this relation holds in each interval $(c,d)$ as before, one can easily conclude by recalling that, being
$u\in H^1$, it holds
\[
\int_{\{u(t)=0\}}|\dot u(t)|\,dt = 0.\qedhere
\]
\end{proof}
Let
\[
s(t)= \sum_{k=0}^{n-1} y_k \chi_{[t_k,t_{k+1})}(t)
\]
denote a simple function. We define the quantity
\begin{equation}\label{defin delta}
\d(s):= \inf \{ t_{k+1}-t_k: \ k=0,\ldots,n-1 \}.
\end{equation}
Given any measurable function $u \in \mathcal{M}(0,1)$, it is well known that for every $\eps>0$ there is a simple function $s_u$ such that $\| u-s_u\|_{\infty}<\eps$. In general the quantity $\d(s_u)$ depends on $u$ and $\eps$. The following Lemma says that if we consider the family of the minimizers $\{\wh{u}_{(a,b)}\}$, given $\eps>0$ it is possible to find a family of approximating simple functions $\{s_{(a,b)}\}$ such that $\d(s_{(a,b)})$ is bounded below uniformly with respect to $(a,b)$.
\begin{lemma}\label{uniform simple}
For every $\eps>0$, let $m\in\N$ be such that $m > ( \|g\|_{\infty}+\|p\|_{\infty})/\eps$. Then for every $(a,b)\subset\R$
\[
s_{(a,b)}(t):= \sum_{k=0}^{m-1} \wh{u}_{(a,b)}\left(\frac{k}{m}\right) \chi_{\left[\frac{k}{m},\frac{k+1}{m}\right)}(t)
\]
is such that
\[
\| \wh{u}_{(a,b)}- s_{(a,b)}\|_{\infty} < \eps \quad \text{and} \quad \d(s_{(a,b)}) = \bar \d
:= \frac1m.
\]
In particular, $m$ can be chosen only depending on $\eps$ and $\|p\|_\infty$, and not on $p$.
\end{lemma}
\begin{proof}
For every $t \in (0,1)$ there exists $k \in \{0,\ldots, m-1\}$ such that $t \in \left[k/m,(k+1)/m\right)$, so that by Lemma \ref{bound norm minimizers}
\[
| \wh{u}_{(a,b)}(t)- s_{(a,b)}(t)| = \left| \int_{\frac{k}{m}}^t \dot{\wh{u}}_{(a,b)}(\tau)\,d\tau \right| \le \frac{1}{m} \left( \|g\|_{\infty}+\|p\|_{\infty}\right) \qquad \forall t \in (0,1). \qedhere
\]
\end{proof}
\section{The boundary value problem for large intervals}\label{sec: BVP large}
Here and in the next section we consider the minimizer $u_{(a,b)}$ as function of $a,b$ and $p$. For this reason, we write
\begin{itemize}
 \item $u(\cdot\,;a,b;p)$ and $\wh{u}(\cdot\,;a,b;p)$ instead of $u_{(a,b)}$ and $\wh{u}_{(a,b)}$ respectively,
\item $J_{(a,b),p}$ and $\wh{J}_{(a,b),p}$ instead of $J_{(a,b)}$ and $\wh{J}_{(a,b)}$ respectively,
\item $\varphi^+(a,b;p)$ instead of $\varphi^+(a,b)$,
\end{itemize}
to emphasize the dependence we are considering.
As we have already mentioned, we can introduce an auxiliary problem which carries the asymptotic behaviour of \eqref{scaled problem} for $b-a \to +\infty$. Let us consider
\beq\label{limit problem}
\begin{cases}
\ddot{w}(t)=-\left(g_+ - A(p)\right)=: -k & \text{in $(0,1)$}\\
w(0)=0=w(1),
\end{cases}
\eeq
with $k>0$ thanks to \eqref{h2}. Of course, this problem has the unique solution
\begin{equation}\label{sol al pb limite}
w_k(t)=\frac{k}{2}t(1-t).
\end{equation}
The related action functional is
\begin{equation}\label{J infty}
J^\infty_{k}(w):=\int_0^1 \left[\frac{1}{2} \dot{w}^2(t)-k w(t)\right]\,dt,
\end{equation}
which has the unique minimizer $w_k$ in $H_0^1(0,1)^+$ (the uniqueness follows from the strict convexity of $J^\infty_k$). A direct computation gives
\[
J^\infty_{k}(w_k)=- \frac{k^2}{24}.
\]
Having in mind to compare minimizers related to different forcing terms, for any $p$ satisfying \eqref{h2} it is convenient to introduce a subset $\mathcal{P}$ of $L^\infty(\R)$ such that the mentioned threshold can be chosen independently of $q \in \mathcal{P}$. To this aim, first of all we recall the following result.
\begin{lemma}[{\cite[Lemma 2.2]{Or}}]\label{lem: decomposition}
Let $p$ satisfy \eqref{h2}. For every $\eps>0$ there exists a decomposition $p=p_{1,\eps} + \dot{p}_{2,\eps}$, where $\|p_{1,\eps}-A(p)\|_\infty<\frac{\eps}{2}$ and $p_{2,\eps} \in L^\infty(\R)$.
\end{lemma}
This means that if $p$ has asymptotic average it can be written as a sum between a term $p_{1,\eps}$ which is arbitrarily close to the average $A(p)$, plus a term $\dot{p}_{2,\eps}$ which has bounded primitive.

Given $p \in L^\infty(\R)$, we compute $\|p\|_\infty$ and $A(p)$, and for any $0<\eps < 1$ we consider a decomposition as in Lemma \ref{lem: decomposition}; we introduce
\[
M_1:= \|p\|_\infty +1 \quad \text{and} \quad M_\eps:= \|p_{2,\eps}\|_\infty + 1.
\]
We define
\begin{equation}\label{def di P}
\mathcal{P}:= \left\{ q \in L^\infty(\R)\left| \begin{array}{l}
 \|q\|_\infty<M_1, \ \text{$q$ has asymptotic average,} \\ \text{$A(q)= A(p)$, and for any $\eps \in (0,1)$} \\
 \text{there exists a decomposition $q=q_{1,\eps}+\dot{q}_{2,\eps}$} \\
 \text{as in Lemma \ref{lem: decomposition}, with $\|q_{2,\eps}\|_\infty < M_\eps$}
\end{array} \right.\right\}.
\end{equation}
\begin{remark}\label{rem: su P}
Note that given any $p$ satisfying assumption \eqref{h2} we can define the set $\mathcal{P}$, which definition depends on $p$. Clearly, $p \in \mathcal{P}$ and the constant function $A(p)$ belongs to $\mathcal{P}$. Moreover, if $q$ is of type
\[
q(t)= A(p) + \dot{q}_2(t) \quad \text{or} \quad q(t) = p(t) + \dot{q}_2(t),
\]
with $\|q_2\|_\infty, \|\dot{q}_2\|_\infty < 1$, then $q \in \mathcal{P}$.
\end{remark}
We are ready to show that problem \eqref{limit problem} is the limit problem of \eqref{positive BVP} as $b-a \to +\infty$, in the following sense.
\begin{proposition}\label{key proposition}
Let $p$ satisfy assumption \eqref{h2}, and let $\mathcal{P}$ be defined by \eqref{def di P}. For every $0<\eps< \frac{(g_+ - A(p))^2}{24}$ there exists $L_1>0$ depending only on $\eps$ such that if $b-a \ge L_1$ then
\[
-\underline{\alpha} \le \wh{J}_{(a,b),q}(\wh{u}(\cdot\,;a,b;q))  \le - \overline{\alpha} \qquad \forall q \in \mathcal{P},
\]
where
\begin{equation}\label{alfasottotsopra}
\underline{\alpha} := \frac{(g_+ - A(p))^2}{24}+\eps
\quad\text{and}\quad
\overline{\alpha} := \frac{(g_+ - A(p))^2}{24} -\eps.
\end{equation}
\end{proposition}
\begin{remark}
The upper bound on $\eps$ implies that $\wh{u}(\cdot\,;a,b;q)$ cannot vanish identically whenever $b-a>L_1$.
\end{remark}
To prove Proposition \ref{key proposition} we need some intermediate results.
\begin{lemma}\label{approximation linear term}
Let $\mathfrak{F} \subset H_0^1(0,1)^+$ be such that
\[
\|u\|_{L^1(0,1)}\leq M \qquad \forall u \in\mathfrak{F}.
\]
For every $\eps>0$ there exists $L_2=L_2(\eps)>0$ such that, if $b-a>L_2$, then
\begin{align*}
\left| \int_0^1 \left[\frac{1}{(b-a)^2} G((b-a)^2 u)-g_+ u\right] \right| &<\eps  \\
\left| \int_0^1 \left[ g((b-a)^2 u)u -g_+ u\right] \right| & < \eps \\
\left| \int_0^1 \left[ g((b-a)^2 u) u  - \frac{1}{(b-a)^2} G((b-a)^2 u) \right] \right| & < \eps,
\end{align*}
for every $u \in \mathfrak{F}$.
\end{lemma}
\begin{proof}
Let $K_1:= 2(1+M g_+)$ and $\eps>0$ be fixed. By assumption \eqref{h1} we infer the existence of $\bar s>0$ such that
\[
s>\bar s \quad \Longrightarrow \quad\left(1 -\frac{\eps}{K_1}\right) g_+ \le \frac{G(s)}{s}\le g_+.
\]
For every $(a,b)$ and for every $u \in \mathfrak{F}$ we can write
\begin{multline}\label{eq1}
\int_0^1 \frac{G((b-a)^2 u)}{(b-a)^2} = \int_{\{(b-a)^2 u \le \bar s\}} \frac{G((b-a)^2 u)}{(b-a)^2 u} u\\
+\int_{\{(b-a)^2 u > \bar s\}} \frac{G((b-a)^2 u)}{(b-a)^2 u} u.
\end{multline}
As far as the first integral on the right hand side is concerned, since $s> 0$ implies
$0\leq G(s)/s \le g_+$, it results
\beq\label{eq2}
0\le \int_{\{(b-a)^2u \le \bar s\}} \frac{G((b-a)^2 u)}{(b-a)^2 u} u
\le\int_{\{(b-a)^2u \le \bar s\}}  g_+ u \le \frac{g_+ \bar s}{(b-a)^2}< \frac{\eps}{K_1},
\eeq
whenever $b-a>L_2$ sufficiently large, for every $u \in \mathfrak{F}$. 
Note also that the same choice of $L_2$ gives
\[
b-a>L_2 \quad  \Longrightarrow  \quad 0\leq  g_+ \left(\int_0^1 u - \int_{\{(b-a)^2 u > \bar s\}} u \right) < \frac{\eps}{K_1} \qquad \forall u \in \mathfrak{F}.
\]
Let us consider the second integral on the right hand side of \eqref{eq1}. Our choice of $\bar s$ and the previous relation imply that, if $b-a>L_2$, then
\begin{multline*}
-\left(1-\frac{\eps}{K_1}\right)\frac{\eps}{K_1}+g_+\left(1 -\frac{\eps}{K_1}\right) \int_0^1 u \le g_+\left(1 -\frac{\eps}{K_1}\right) \int_{\{(b-a)^2 u > \bar s\}} u \\
\le \int_{\{(b-a)^2 u > \bar s\}} \frac{G((b-a)^2 u)}{(b-a)^2 u} u \le g_+ \int_{\{(b-a)^2 u > \bar s\}} u\le g_+ \int_0^1 u,
\end{multline*}
for every $u \in \mathfrak{F}$. Due to the boundedness of the family $\mathfrak{F}$ in $L^1(0,1)$, it results
\beq\label{eq3}
0  \le  g_+ \int_0^1 u- \int_{\{(b-a)^2 u> \bar s\}} \frac{G((b-a)^2 u)}{(b-a)^2 u} u \le (1+Mg_+)\frac{\eps}{K_1} = \frac{\eps}{2},
\eeq
for every $u \in \mathfrak{F}$. Collecting together \eqref{eq1}, \eqref{eq2} and \eqref{eq3}, we obtain the first estimate of the thesis. To prove the second one, we can adapt the same argument because of assumption \eqref{h1}. The third estimate follows easily.
\end{proof}
\begin{lemma}\label{estimate forcing term}
Let $\mathfrak{F} \subset H_0^1(0,1)^+$ be such that
\[
\|u\|_{L^1(0,1)}\leq M \qquad \forall u \in\mathfrak{F}.
\]
For $\eps>0$, $\delta_1>0$ and for every $u \in \mathfrak{F}$, let us assume the existence of a simple function $s_u$ such that
\[
\|u-s_u\|_{\infty}<\eps_1 \quad\text{ and }\quad \d(s_u)\geq\d_1,
\]
where $\delta(\cdot)$ is defined as in \eqref{defin delta} and $\eps_1:=\eps/(M_1+M+\|g\|_{\infty} + 1)$. Then there exists $L_3>0$, depending on $\eps$, $\delta_1$ but independent of $q \in \mathcal{P}$, such that, if $b-a>L_3$, then
\[
\left| \int_0^1 \left( \wh{q}_{(a,b)} - A(p) \right) u \right| <\eps,
\]
for every $u \in \mathfrak{F}$ and $q \in \mathcal{P}$.
\end{lemma}
\begin{proof}
Let $K_2:=(M_1+M+\|g\|_{\infty} + 1)$, and let us assume that $(a,b)=(0,L)$ to ease the notation. It is straightforward to apply the following argument for a general $(a,b) \subset \R$. Let us consider, for $(c,d) \subset [0,1]$,
\begin{align*}
\int_c^d \wh{q}_L(t)\,dt & =\frac{1}{L} \int_{cL}^{dL} q(t)\,dt = \frac{d-c}{L(d-c)} \int_{cL}^{dL} q(t)\,dt .
\end{align*}
For any $\eps>0$ sufficiently small, we consider the decomposition $q=q_{1,\eps}+\dot{q}_{2,\eps}$ given by Lemma 	 \ref{lem: decomposition}. By definition of $\mathcal{P}$, we know that
\begin{multline*}
\sup_{t \in \R} \left| \frac{1}{T} \int_t^{t+T} q(\sigma )\,d \sigma -A(p) \right| \\
\le
\sup_{t \in \R}  \left(\frac{1}{T} \int_t^{t+T} \left|q_{1,\eps}(\sigma) -A(p) \right| \,d \sigma + \left|\frac{1}{T} \int_t^{t+T} \dot{q}_{2,\eps}(\sigma )\,d\sigma \right| \right) \\
< \frac{\eps}{2} + \frac{2}{T} \|q_{2,\eps}\|_\infty < \frac{\eps}{2} + \frac{2}{T}M_\eps <\eps
\end{multline*}
whenever $T>\bar T(\eps) :=4M_\eps /\eps$, independently of $q \in \mathcal{P}$. Therefore, if $(d-c)L > \bar T\left(\eps/K_2\right)$, then
\[
\left| \frac{1}{L(d-c)}\int_{Lc}^{Ld} q(t) \,dt - A(p) \right|<  \frac{\eps}{K_2} \qquad \forall q \in \mathcal{P}.
\]
Let us consider the family of simple functions $\{s_u: u \in \mathfrak{F}\}$. Let us set $L_3:= (1/\delta_1) \bar T\left(\eps/ K_2 \right)$; for $s_u =\sum_{k=0}^{n-1} y_k \chi_{[t_k,t_{k+1})}$, we note that if $L>L_3$, then
\[
(t_{k+1}-t_k)L \ge \delta_1 L_3 =\bar T\left(\frac{\eps}{K_2}\right),
\]
so that
\begin{align*}
\left| \int_0^1 \left( \wh{q}_L-A(p) \right) s_u \right| & \le \sum_{k=0}^{n-1} | y_k |(t_{k+1}-t_k) \left| \frac{1}{L(t_{k+1}-t_k)}\int_{L t_k}^{L t_{k+1}} q(\sigma) \,d \sigma -A(p) \right| \\
& < \frac{\eps}{K_2} \int_0^1 |s_u|< \frac{\eps}{K_2}  (M+1),
\end{align*}
independently of $u \in \mathfrak{F}$ and on $q \in \mathcal{P}$, where for the last inequality we use the boundedness of $\mathfrak{F}$ in $L^1(0,1)$. Therefore, if $L \ge L_3$, then
\begin{align*}
\left| \int_0^1 \left( \wh{q}_L -A(p) \right) u \right| & \le \int_0^1 \left|\wh{q}_L+A(p) \right| \left|  u-s_u \right|  + \left| \int_0^1 \left(\wh{q}_L-A(p)\right) s_u\right|  \\
&< \left( \|q\|_{\infty} +\|g\|_{\infty} \right) \|u-s_u\|_{\infty} + \frac{\eps}{K_2} (M+1) \\
& < \eps,
\end{align*}
for every $u \in \mathfrak{F}$ and for every $q \in \mathcal{P}$ (for the reader's convenience, we recall that by definition $M_1 > \|q\|_\infty$ for every $q \in \mathcal{P}$).
\end{proof}
We are in position to prove Proposition \ref{key proposition}.
\begin{proof}[Proof of Proposition \ref{key proposition}]
Let us consider the family
\[
\mathfrak{F}:=\{\wh{u}(\cdot\,;a,b;q): (a,b) \subset \R, q \in \mathcal{P}\} \cup \{w_{(g_+-A(p))}\},
\]
where we recall that $\wh{u}(\cdot;a,b;q)$ is the minimizer of $\wh{J}_{(a,b),q}$ (defined by \eqref{scaled functional}), and $w_{(g_+-A(p))}$ has been defined by \eqref{sol al pb limite}. In light of Lemmas \ref{bound norm minimizers} and \ref{uniform simple}, the family satisfies the assumptions of Lemmas \ref{approximation linear term} and \ref{estimate forcing term}.

Let $L_1:= \max\left\{L_2\left(\eps/2\right),L_3\left( \eps/2\right)\right\}$, where $L_2$ and $L_3$ have been defined in the quoted statements, and we recall that $L_3$ is independent of $q \in \mathcal{P}$. By definition, if $b-a>L_1$, then
\[
\begin{split}
\wh{J}_{(a,b),q}(\wh{u}(\cdot\,;a,b;q)) &> \int_0^1\left[ \frac{1}{2} \dot{\wh{u}}\,^2(t;a,b;q)-\left(g_+ - A(p)\right)\wh{u}(t;a,b;q)\right]\,dt - \eps \\
& \ge \inf_{H_0^1(0,1)^+} J^\infty_{\left(g_+ -A(p)\right)} - \eps= -\frac{\left(g_+ - A(p)\right)^2}{24}-\eps,
\end{split}
\]
for every $q \in \mathcal{P}$, where we recall that $J^\infty_k$ has been defined in \eqref{J infty} for any $k \in \R$. Moreover, by minimality,
\[
\begin{split}
\wh{J}_{(a,b),q}(\wh{u}(\cdot\,;a,b;q)) &\le  \wh{J}_{(a,b),q}(w_{(g_+ -A(p))}) \\
& < \int_0^1\left[ \frac{1}{2} \dot{w}_{(g_+ -A(p))}^2(t)-\left(g_+ -A(p)\right)w_{(g_+ - A(p))}(t) \right]\,dt +\eps \\
& = \inf_{H_0^1(0,1)^+} J^\infty_{\left(g_+-A(p)\right)}+\eps= -\frac{\left(g_+-A(p)\right)^2}{24}+\eps,
\end{split}
\]
whenever $b-a > L_1$.
\end{proof}

Now we can come back on the time interval $[a,b]$: due to the explicit relations \eqref{scaling} and \eqref{scaled functional}, we can summarize the previous results in the following statement.
\begin{corollary}\label{indietro su (a,b)}
For $0<\eps<(1-A(p))^2/24$, let $L_1(\eps)$ be defined as in Proposition \ref{key proposition}. If $b-a>L_1(\eps)$ then
\[
-\underline{\a} (b-a)^3 \le \f^+(a,b;q) \le -\overline{\a}(b-a)^3,
\]
for every $q \in \mathcal{P}$, where $\underline\alpha$, $\overline\alpha$ are defined as in equation \eqref{alfasottotsopra}.
\end{corollary}
\begin{remark}\label{su L1}
By definition, $L_1 \ge L_2,L_3$. Therefore, if $b-a>L_1$, Lemmas \ref{approximation linear term} and \ref{estimate forcing term} hold true; in particular, we deduce that for every $0<\eps<\frac{1}{24}(1-A(p))^2$, if $b-a>L_1(\eps)$, then
\[
 \left| \int_a^b \left[g\left( u(t;a,b;q)\right)u(t;a,b;q) -G\left(u(t;a,b;q)\right)\right]\,dt\right| < \eps (b-a)^3
\]
for every $q \in \mathcal P$.
\end{remark}
In the next statement and in the rest the symbol $\| \cdot\|$ denotes the Dirichlet $H_0^1$ norm on the considered interval, that is,
\[
\|u\| = \left(\int_a^b \dot{u}^2(t)\,dt\right)^{1/2} \qquad  \forall u \in H_0^1(a,b).
\]
\begin{corollary}\label{lem:stime su norme}
There exists $L_4>0$ and a positive constant $C_1>0$ such that, if $b-a \ge L_4$, then $\| u(\cdot\,;a,b;q)\| \ge C_1(b-a)^{3/2}$ and $\| u(\cdot\,;a,b;q)\|_{\infty} \ge C_1(b-a)^2$ for every $q \in \mathcal{P}$.
\end{corollary}
\begin{proof}
Since the function $\lambda \mapsto J_{(a,b),q}( \lambda u(\cdot\,;a,b;q))$ reaches its minimum at $\lambda=1$, it results
\[
\int_a^b \left[\dot{u}^2(t;a,b;q)- g\left(u(t;a,b;q)\right)u(t;a,b;q) + q(t)u(t;a,b;q) \right]\,dt= 0.
\]
We can solve this identity for the last term and substitute in the expression of $J_{(a,b),q}(u(\cdot;a,b;q))$:
\begin{multline*}
J_{(a,b),q}( u(\cdot\,;a,b;q))= -\int_a^b \frac{1}{2}\dot{u}^2(t;a,b;q)\,dt \\
+ \int_a^b \left[g\left(u(t;a,b;q) \right)u(t;a,b;q)  - G\left(u(t;a,b;q)\right)\right] \,dt.
\end{multline*}
Given $\eps>0$ sufficiently small, if $b-a > L_1(\eps)$ defined in Proposition \ref{key proposition}, we have (we refer also to Corollary \ref{indietro su (a,b)} and to Remark \ref{su L1})
\begin{multline*}
J_{(a,b),q}( u(\cdot\,;a,b;q)) > -\frac{1}{2} \| \dot{u}(\cdot\,;a,b;q) \|^2 - \eps(b-a)^3 \quad \text{and} \\
 J_{(a,b),q}(u(\cdot\,;a,b;q))  \le  \left(-\frac{(g_+ -A(p))^2}{24} +\eps\right)(b-a)^3,
\end{multline*}
for every $q \in \mathcal{P}$, from which we deduce
\[
\| \dot{u}(\cdot\,;a,b;q) \|^2 > \left(\frac{(g_+ -A(p))^2}{12} - 4 \eps\right)(b-a)^3  \qquad \forall q \in \mathcal{P}.
\]
We choose $\bar \eps= (g_+-A(p))^2/96$ and set $L_4=L_1(\bar \eps)$. Hence
\[
\|u(\cdot\,;a,b;q) \|  \ge \frac{(g_+-A(p))}{\sqrt{24}} (b-a)^{\frac{3}{2}}  \qquad \forall q \in \mathcal{P},
\]
and
\begin{align*}
\frac{(g_+ -A(p))^2}{24}  (b-a)^3 & \le \int_a^b \dot{u}^2(t;a,b;q)\,dt \\
 &= \int_a^b \left[g\left(u(t;a,b;q)\right)u(t;a,b;q) - q(t)u(t;a,b;q) \right]\,dt \\
& \le \left(\|g\|_{\infty} +M_1\right) \|u(\cdot;a,b;q)\|_{\infty} (b-a),
\end{align*}
which gives the desired result for
\[
C_1:= \frac{(g_+ -A(p))^2}{24(\|g\|_{\infty}+M_1)}. \qedhere
\]
\end{proof}
Finally, we can prove that if $b-a$ is sufficiently large, then any minimizer $u(\cdot\,;a,b;q)$ with $q \in \mathcal{P}$ is an actual solution of the boundary problem \eqref{positive BVP}.
\begin{proposition}[Existence]\label{prop:existence}
Let $p$ satisfy assumption \eqref{h2}, and let $\mathcal{P}$ be defined by \eqref{def di P}. There exists $\tilde L \ge L_4$ such that, if $b-a \ge \tilde L$, then $u(t;a,b;q) > 0$ for every $t\in(a,b)$, $q \in \mathcal{P}$. Hence, $u(\cdot\,;a,b;q)$ is a solution of \eqref{positive BVP}.
\end{proposition}
\begin{proof}
For $q \in \mathcal{P}$, let
\[
\{t \in (a,b): u(t;a,b;q)>0\} = \bigcup_{i \in I} (a_i,b_i),
\]
where $I$ is a family of indexes and $u(t;a,b;q)>0$ for $t \in (a_i,b_i)$ (thus the $(a_i,b_i)$ are disjoint intervals). By continuity, there exists $j \in I$ such that in $(a_j,b_j)$ there exists a point $\tau$ of global maximum for $u(\cdot;a,b;q)$. By Corollary \ref{lem:stime su norme}, we know that $u(\tau;a,b;q) \ge C_1 (b-a)^2$ whenever $b-a \ge L_4$, for every $q \in \mathcal{P}$. Assume by contradiction that $(a_j,b_j) \neq (a,b)$; say, for instance, $a_j>a$. In order to obtain a contradiction, we consider separately the cases $A(p)>0$ or $A(p)\leq 0$.

\paragraph{The case $A(p) > 0$.} We choose $0<\eps<\min\left\{C_1,2A(p)/3\right\}$, where we recall that $C_1$ has been defined in Corollary \ref{lem:stime su norme}, and we consider the decomposition of Lemma \ref{lem: decomposition} for the forcing term $q$. By the monotonicity of $g$, assumption \eqref{h1}, there exists $s_\eps:= g^{-1}\left(A(p)-3\eps/2 \right)$. Assuming $b-a$ sufficiently large in such a way that $u(\tau)>s_\eps$ we
can introduce
\[
\begin{split}
T &:= \inf \left\{ \bar t> a_j: u(t;a,b;q)>s_\eps \text{ for every $t \in (\bar t, \tau)$}\right\},\\
a'&:= \inf \left\{\bar t\leq T: \dot{u}(t;a,b;q)\geq 0 \text{ for every $t \in [\bar t,T]$} \right\}
\end{split}
\]
(in particular, if $\dot{u}(T;a,b;q)=0$ then $a':=T$). Note that, by definition,
\begin{equation}\label{eq8}
\begin{cases}
0 \le u(t;a,b;q) \le s_\eps & \text{if $t \in [a',T]$} \\
u(t;a,b;q) \ge s_\eps & \text{if $t \in [T,\tau]$}.
\end{cases}
\end{equation}
As $u(\cdot\,;a,b;q) \in \mathcal{C}^1(a,b)$, $a'\geq a_j >a$ necessarily implies $\dot{u}(a';a,b;q)=0$. As a consequence, if we reach a contradiction, we deduce that both $a' = a_j = a$ and $\dot{u}(a;a,b;q)>0$.

\paragraph{Step 1)} \emph{there exists $C_2>0$ independent of $q \in \mathcal{P}$ such that $T-a' \le C_2$}.\\
By the monotonicity of $g$ and \eqref{eq8}, we deduce that, for every $t \in (a',T)$,
\begin{align*}
\ddot{u}(t;a,b;q) &= -g(u(t;a,b;q))+q_{1,\eps}(t) + \dot{q}_{2,\eps}(t) \\ & \ge -g(s_\eps)+ A(p)-\frac{\eps}{2} +\dot{q}_{2,\eps}(t) =\eps + \dot{q}_{2,\eps}(t).
\end{align*}
By integrating twice in $(a',t)$, and using the fact that $\dot{u}(a';a,b;q)=0$, we obtain
\[
s_\eps \ge u(T;a,b;q) - u(a';a,b;q) \ge \frac{\eps}{2}(T-a')^2 - 2 M_\eps (T-a')  ,
\]
which provides the desired estimate.

\paragraph{Step 2)} \emph{There exists $C_3>0$ independent of $q \in \mathcal{P}$ such that $\dot{u}(T;a,b;q) \le C_3$}. \\
As $g(s) \ge 0$ for $s \ge 0$, we see that, for every $t \in (a',T)$,
\begin{align*}
\ddot{u}(t;a,b;q) &= -g(u(t;a,b;q))+q_{1,\eps}(t) + \dot{q}_{2,\eps}(t) \\ & \le A(p) + \frac{\eps}{2} +\dot{q}_{2,\eps}(t).
\end{align*}
By integrating in $(a',T)$, we deduce that
\[
\dot{u}(T;a,b;q) \le \left(A(p) + \frac{\eps}{2}\right)(T-a') + 2M_\eps \le C_3,
\]
where we use the first step and the fact that $\dot{u}(a';a,b;q)=0$.

\paragraph{Step 3)} \emph{Conclusion of the proof in case $A(p)>0$}.\\
By the monotonicity of $g$ (assumption \eqref{h1}) and \eqref{eq8}, we deduce that, for every $t \in (T,\tau)$,
\begin{align*}
\ddot{u}(t;a,b;q) &= -g(u(t;a,b;q))+q_{1,\eps}(t) + \dot{q}_{2,\eps}(t) \\ & \le -g(s_\eps) + A(p) +\frac{\eps}{2} +\dot{q}_{2,\eps}(t) =2\eps + \dot{q}_{2,\eps}(t).
\end{align*}
By integrating twice in $(T,t)$ and evaluating in $\tau$, we deduce
\begin{align*}
u(\tau;a,b;q) &\le \eps (b-a)^2 + \left(\dot{u}(T;a,b;q) + 2 M_\eps \right) (b-a) + u(T;a,b;q) \\
& \le \eps(b-a)^2 + \left(C_3 +2 M_\eps\right)(b-a) + s_\eps,
\end{align*}
where we used the result of the previous step and the definition of $T$. The choice $\eps<C_1$ gives a contradiction with Corollary \ref{lem:stime su norme} for $b-a$ sufficiently large (greater than a constant $\tilde L$ depending only on $\mathcal{P}$ and not on the particular choice of $q$).

\paragraph{The case $A(p)\leq0$.} We choose $0<\eps< C_1$, where we recall that $C_1$ has been defined in Corollary \ref{lem:stime su norme}, and consider the decomposition of Lemma \ref{lem: decomposition} for the forcing term $q$. For every $t \in (a_j,b_j)$ we have
\[
\ddot{u}(t;a,b;q)  = -g\left(u(t;a,b;q)\right) +q_{1,\eps}(t) + \dot{q}_{2,\eps}(t) \le \frac{\eps}{2} + \dot{q}_{2,\eps}(t),
\]
where we used the fact that $g(s) \ge 0$ for $s \ge 0$. By integrating twice in $(a_j,t)$ with $t \in (a_j,b_j)$, and evaluating in $\tau$,  we obtain
\[
u(\tau;a,b;q) \le \eps(b-a)^2 + 2 M_\eps (b-a).
\]
Having chosen $\eps<C_1$, this immediately contradicts Corollary \ref{lem:stime su norme} for $b-a$ sufficiently large.
\end{proof}

For the results of the next sections it is important to prove the uniqueness of the minimizer of the functional $J_{(a,b),q}$ with $q \in \mathcal{P}$. In light of the previous and the next statements, this uniqueness is guaranteed provided $b-a>\tilde L$. In the following proposition the forcing term $p$ is fixed; therefore, we will use the simplified notation of the previous section.

\begin{proposition}[Uniqueness]\label{prop:uniqueness}
Let $u$ and $v$ be functions in $\mathcal{C}^2(a,b) \cap H_0^1(a,b)$ such that $u>0$ and $v>0$ in $(a,b)$. Assume that
\[
J_{(a,b)}(u)=J_{(a,b)}(v) = \varphi^+(a,b).
\]
Then $u \equiv v$ in $[a,b]$.
\end{proposition}
\begin{proof}
Let us consider the function
\[
\Phi(\lambda):= J_{(a,b)}((1-\lambda) u + \lambda v).
\]
We note that $\Phi \in \mathcal{C}^1(\R)$ and
\[
\Phi'(\lambda)= dJ_{(a,b)}( (1-\lambda) u + \lambda v)[v-u].
\]
As $\Phi(0) = \Phi(1)$, there exists $\bar \lambda \in (0,1)$ such that $\Phi'(\bar \lambda)=0$, that is,
\begin{equation}\label{eq4}
\int_a^b \left[\left(1-\bar \lambda\right) \dot u+\bar \lambda \dot v\right](\dot v- \dot u) - g\left(\left(1-\bar \lambda\right) u+ \bar \lambda v\right)(v-u) + p (v-u) = 0.
\end{equation}
Also, by minimality we know that $\Phi'(0)=\Phi'(1)=0$, that is
\begin{align}
& \int_a^b \dot u(\dot v-\dot u) -g(u)(v-u)+p(v-u)=0  \label{eq5}\\
& \int_a^b \dot v(\dot v-\dot u) -g(v)(v-u) + p(v-u) = 0. \label{eq6}
\end{align}
If we consider \eqref{eq4} and subtract $(1-\bar \lambda)$ times  \eqref{eq5} and $\bar \lambda$ times  \eqref{eq6}, we obtain
\begin{equation}\label{eq7}
\int_a^b \left[ \left(1-\bar \lambda\right) g(u) + \bar \lambda g(v) - g\left(\left(1-\bar \lambda\right) u + \lambda v\right) \right](v-u)=0.
\end{equation}
We claim that
\begin{equation}\label{claim 1}
\text{either $u \equiv v$ or the function $v-u$ changes sign in $(a,b)$}.
\end{equation}
Indeed, assume $u \not \equiv v$ and, w.l.o.g., $v \ge u$ in $(a,b)$. The set $A:= \{t \in (a,b): v(t)>u(t)\}$ is not empty and has positive measure. Hence, by \eqref{eq7} and the strict concavity of $g$ in $(0,+\infty)$, assumption \eqref{h1}, we deduce that
\begin{align*}
0&=\int_a^b \left[ \left(1-\bar \lambda\right) g(u) + \bar \lambda g(v) - g\left(\left(1-\bar \lambda\right) u + \lambda v\right) \right](v-u) \\
&= \int_A \left[ \left(1-\bar \lambda\right) g(u) + \bar \lambda g(v) - g\left(\left(1-\bar \lambda\right) u + \lambda v\right) \right](v-u) < 0,
\end{align*}
a contradiction. This proves the claim \eqref{claim 1}, so that it remains to show that $v-u$ cannot change sign in $(a,b)$. By contradiction again, assume that $v-u$ changes sign in $(a,b)$, so that in particular there exists $\tau$ in $(a,b)$ such that $u(\tau)=v(\tau)$. Say, for instance, that
\[
\int_a^\tau \left( \frac{1}{2}\dot{u}^2- G(u) + p u\right) \le \int_a^\tau \left( \frac{1}{2}\dot{v}^2- G(v) + pv\right);
\]
necessarily it results
\[
\int_\tau^b \left( \frac{1}{2}\dot{u}^2- G(u) + p u\right) \ge \int_\tau^b \left( \frac{1}{2}\dot{v}^2- G(v) + pv\right).
\]
Let
\[
\tilde u(t):= \begin{cases} u(t) & \text{if $t \in (a,\tau)$} \\
v(t) & \text{if $t \in [\tau,b)$}. \end{cases}
\]
By definition $\tilde u \in H_0^1(a,b)^+$, $\tilde u>0$ in $(a,b)$ and $J_{(a,b)}(\tilde u) \le J_{(a,b)}(u)= \varphi^+(a,b)$, that is, $\tilde u$ is a minimizer of $J_{(a,b)}$ in $H_0^1(a,b)^+$ which is strictly positive in $(a,b)$; hence, it solves the boundary value problem \eqref{positive BVP} and has to be of class $\mathcal{C}^2(a,b)$. This implies that $\dot u(\tau)=\dot v(\tau)$, and recalling that $u(\tau)= v(\tau)$, we can apply the uniqueness theorem for the initial value problems, proving that $u \equiv v$ in $(a,b)$.
\end{proof}

Let $p \in \mathcal{P}$, and let $\mathcal{P}$ be defined by \eqref{def di P}. Collecting together the results of Propositions \ref{prop:existence} and \ref{prop:uniqueness}, we can conclude that there exists $\tilde L>0$ such that for every $(a,b)\subset \R$ with $b-a \ge \tilde L$ and for every $q \in \mathcal{P}$ there exists a unique minimizer $u(\cdot\,;a,b;q)$ of the functional $J_{(a,b),q}$ in $H_0^1(a,b)^+$, which is strictly positive in $(a,b)$ and hence solves problem \eqref{positive BVP} with forcing term $q$. It is then possible to define a map which associates to each triple $(a,b,q)$, with $b-a \ge \tilde L$ and $q \in \mathcal{P}$, the unique minimizer $u(\cdot;a,b;q)$. We conclude this section proving that this map is continuous.

\begin{lemma}\label{lem: continuous dependence}
Let $p$ satisfy \eqref{h2}, and let $\mathcal{P}$ be defined by \eqref{def di P}. Let $A$ and $B$ be fixed and let
\[
 \mathcal I := \left\{(t,a,b)\in \R^3: b-a > \tilde L, \  A  < a \le t \le b < B\right\},
\]
where $\tilde L$ has been defined in Proposition \ref{prop:existence}. Let us consider the metric space $\mathcal{P}$ endowed with the distance $d(q_1,q_2)= \|q_1-q_2\|_{L^2(A,B)}$. The map
\[
 (t,a,b,q) \in \overline{\mathcal{I}} \times \mathcal{P} \mapsto \left(u(t;a,b;q), \dot{u}(t;a,b;q)\right) \in \R^2
\]
is continuous.
\end{lemma}
\begin{proof}
Let $(a_n,b_n,p_n) \to (a^*,b^*,p^*)$ in $\overline{\mathcal{I}} \times \mathcal{P}$. Thanks to the explicit relations \eqref{scaling}, we can consider the scaled functions $\wh{u}_n:=\wh{u}(\cdot\,; a_n,b_n;\wh{p}_n)$ and $\wh{u}^*:=\wh{u}(\cdot\,; a^*,b^*;\wh{p}^*)$. Having chosen $b-a > \tilde L$ and $(p_n) \subset \mathcal{P}$, from the previous results we deduce that each $\wh{u}_n$ solves problem \eqref{scaled problem} with $a_n$, $b_n$, $\wh{p}_n$ instead of $a$, $b$, $\widehat{p}$. By Lemma \ref{bound norm minimizers}, we know that the sequence $(\wh{u}_n)$ is bounded in $H_0^1(0,1)$, so that, up to a subsequence, it is weakly convergent in $H_0^1(0,1)$ to some $\widetilde{u} \in H_0^1(0,1)^+$. This, together with the fact that, up to a subsequence, $\wh{p}_n \to \wh{p}^*$ almost everywhere in $[0,1]$ (this follows from the convergence of $\wh{p}_n$ to $\wh{p}^*$ in $L^2(0,1)$), implies that $\wh{u}_n \to \widetilde{u}$ in $H^2(0,1) \cap H_0^1(0,1)$, and
\[
\begin{cases}
\ddot{\widetilde{u}}(t) + g\left((b^*-a^*)^2 \widetilde{u}(t) \right) = \widehat{p}^*(t) & t \in (0,1) \\
\widetilde{u}(0)=0=\widetilde{u}(1) \\
\widetilde{u}(t) \ge 0 & t \in (0,1).
\end{cases}
\]
We aim at proving that $\widetilde{u} \equiv \widehat{u}^*$; if this is not true, then the variational characterization of $\wh{u}^*$ and Proposition \ref{prop:uniqueness} imply that
\begin{equation}\label{rel minimizers}
\wh{J}_{(a^*,b^*), p^*}(\widehat{u}^*) < \wh{J}_{(a^*,b^*), p^*}(\widetilde{u}).
\end{equation}
By the continuity of $\widehat{J}$ with respect to $u$, $p$, $a$ and $b$, we have also
\begin{equation}\label{rel limting cont}
\wh{J}_{(a_n,b_n), p_n}(\widehat{u}_n) \to \wh{J}_{(a^*,b^*), p^*}(\widetilde{u}) \quad \text{and} \quad  \wh{J}_{(a_n,b_n), p_n}(\widehat{u}^*) \to  \wh{J}_{(a^*,b^*), p^*}(\widehat {u}^*).
\end{equation}
A comparison between \eqref{rel minimizers} and \eqref{rel limting cont} for $n$ sufficiently large gives a contradiction with the fact $\widehat{u}_n$ reaches the minimum of $\wh{J}_{(a_n,b_n), p_n}$ in $H_0^1(0,1)^+$, so that necessarily $\widetilde{u} \equiv \widehat{u}^*$. Since this argument holds for any subsequence, we deduce the convergence of the whole sequence, and to obtain the desired result it is sufficient to observe that, since $\wh{u}_n \to \widehat{u}^*$ in $H^2(0,1) \cap H_0^1(0,1)$, then $\wh{u}_n \to  \widehat{u}^*$ in $\mathcal{C}^1([0,1])$.
\end{proof}

\section{Non-degeneracy of positive minimizers}\label{sec: non deg}

Assume that $u$ solves \eqref{positive BVP} in $(a,b)$; we can consider the variational equation
\begin{equation}\label{variational eq}
\begin{cases}
\ddot{\psi}(t)+ g'(u(t))\psi(t)=0 & t \in (a,b) \\
\psi(a)=0=\psi(b).
\end{cases}
\end{equation}
\begin{definition}
We say that $u$ is \emph{non-degenerate} as solution of \eqref{positive BVP} if problem \eqref{variational eq} has only the trivial solution $\psi \equiv 0$ in $H^2(a,b) \cap H_0^1(a,b)$.
\end{definition}
The main result of this section is the following.
\begin{proposition}\label{prop: non degeneracy}
Let $p$ satisfy \eqref{h2}, $\mathcal{P}$ be defined by \eqref{def di P}, and $\tilde L$ be defined as in Proposition \ref{prop:existence}, and let us assume that $b-a \ge \tilde L$. The function $u(\cdot;a,b;p)$ is non-degenerate as solution of the boundary value problem \eqref{positive BVP}.
\end{proposition}

For the proof, we will use some known results in singularity theory, which we recall here and for which we refer to Section 3.2 of the book by Ambrosetti and Prodi \cite{AmPr}.

\begin{definition}\label{def: singular}
Let $\Phi:\Omega \subset E \to F$ be of class $\mathcal{C}^2(\Omega)$, where $\Omega$ is open, $E$ and $F$ are Banach spaces and $u_0 \in \Omega$. We say that $u_0$ is \emph{singular} if $d\Phi(u_0)$ is not invertible. It is \emph{ordinary singular} if it is singular and
\begin{itemize}
\item[($i$)] $\Ker\left(d\Phi (u_0)\right)$ is one-dimensional:
\[
\Ker \left(d\Phi (u_0)\right) ) = \R \psi_0 \quad \text{for some $\psi_0 \in E \setminus \{0\}$};
\]
$\im \left(d\Phi (u_0)\right)$ is closed and has codimension $1$:
\[
\im\left(d\Phi (u_0)\right) = \left\{ q \in F: \langle \gamma_0,q \rangle = 0 \right\} \quad \text{with $\gamma_0 \in F^* \setminus \{0\}$}.
\]
\item[($ii$)] $\langle \gamma_0, d^2 \Phi(u_0)[\psi_0,\psi_0] \rangle \neq 0$.
\end{itemize}
\end{definition}

\begin{theorem}[Ambrosetti-Prodi]\label{thm: on ordinary singular}
Let $u_0$ be an ordinary singular point for $\Phi$, and, say,
\[
\langle \gamma_0, d^2 \Phi(u_0)[\psi_0,\psi_0] \rangle > 0;
\]
let $q_0=\Phi(u_0)$, and let $q \in F$ be such that $\langle \gamma_0, q \rangle >0$. Then there exists a neighbourhood $U$ of $u_0$ in $E$ and a positive number $\eps^*$ such that the equation
\[
\Phi(u)=q_0+\eps q, \qquad u \in U
\]
has exactly two solutions for $0<\eps<\eps^*$ and no solution for $-\eps^*<\eps<0$.
\end{theorem}

We are ready to show that $u(\cdot;a,b;p)$ is non-degenerate.

\begin{proof}[Proof of Proposition \ref{prop: non degeneracy}]
Let
\[
X:= H^2(a,b) \cap H_0^1(a,b), \quad \|u\|_X:= \|\ddot{u}\|_2, \quad Y:= L^2(a,b).
\]
We introduce the map $\mathcal{F}:X \to Y$ defined by
\[
\mathcal{F}(u) = -\ddot{u} - g(u).
\]
Under assumption \eqref{h1}, it is immediate to see that $\mathcal{F} \in \mathcal{C}^2(X,Y)$ and
\[
d \mathcal{F}(u)\psi = -\ddot{\psi} -g'(u)\psi, \quad  \quad d^2 \mathcal{F}(u)[\psi_1,\psi_2] = -g''(u) \psi_1\psi_2.
\]
By the Fredholm alternative, $u(\cdot;a,b;p)$ is degenerate as solution of \eqref{positive BVP} if and only if it is singular for $\mathcal{F}$. So, let us assume by contradiction that $u(\cdot;a,b;p)$ is degenerate as solution of \eqref{positive BVP}.

\paragraph{Step 1)} \emph{$u(\cdot;a,b;p)$ is ordinary singular for $\mathcal{F}$}.\\
We have to show that $u(\cdot;a,b;p)$ satisfies points ($i$) and ($ii$) of Definition \ref{def: singular}. By degeneracy, problem
\begin{equation}\label{var eq for p}
\begin{cases}
\ddot{\psi}(t)+ g'(u(t;a,b;p))\psi(t)=0 & t \in (a,b) \\
\psi(a)=0=\psi(b)
\end{cases}
\end{equation}
has a nontrivial solution $\psi_0$, that is, $0$ is an eigenvalue for the operator \\ $d \mathcal{F}(u(t;a,b;p))$; this is a Sturm-Liouville operator with Dirichlet boundary conditions, hence all its eigenvalues are simple, and in particular \\
$\Ker \left(d \mathcal{F}(u(t;a,b;p))\right) = \R \psi_0$. Moreover, in light of the Fredholm alternative, $d \mathcal{F}(u(t;a,b;p))$ is a Fredholm operator with index $0$, so that property ($i$) in Definition \ref{def: singular} follows.

As far as point ($ii$) is concerned, first of all we claim that $0$ is the first eigenvalue of $d \mathcal{F}(u(t;a,b;p))$; if not, there exists $\lambda_1<0$ and $\psi_1 \in X \setminus \{0\}$ such that
\[
\begin{cases}
\ddot{\psi}_1(t)+ g'(u(t;a,b;p))\psi_1(t)=-\lambda_1 \psi_1(t) & t \in (a,b) \\
\psi_1(a)=0=\psi_1(b).
\end{cases}
\]
On the other hand, as $u(t;a,b;p)$ is a local minimizer for $J_{(a,b),p}$, we see that $d^2 J_{(a,b),p}(u(t;a,b;p))$ is a positive semi-definite quadratic form; this implies that
\begin{align*}
0 & \le d^2 J_{(a,b),p}(u(t;a,b;p))[\psi_1,\psi_1] = -\int_a^b\left( \ddot{\psi}_1 + g'(u(t;a,b;p))\psi_1\right) \psi_1 \\
&= \lambda_1 \int_a^b \psi_1^2 <0,
\end{align*}
a contradiction. Having proved that $0$ is the first eigenvalue of $d \mathcal{F}(u(t;a,b;p))$, we can assume that $\psi_0>0$ in $(a,b)$. By the Fredholm alternative, we know that $\im\left( d \mathcal{F} (u(t;a,b;p)) \right) = \{ q \in Y: \langle \gamma_0,q \rangle = 0\}$, where $\langle \gamma_0,q \rangle=\int_a^b \psi_0 q$. Hence
\[
\langle \gamma_0, d^2 \mathcal{F}(u(t;a,b;p))[\psi_0,\psi_0] \rangle = -\int_a^b g''( u(t;a,b;p) ) \psi_0^3 \neq 0
\]
being $g''<0$ in $(0,+\infty)$ and $\psi_0>0$ in $(a,b)$.

\paragraph{Step 2)} \emph{Conclusion of the proof}. \\
By definition, $\mathcal{F}(u(t;a,b;p) ) = p$. We can choose $q \in Y$ such that
\begin{itemize}
\item $\int_a^b q \psi_0 >0$;
\item $p+\eps q \in \mathcal{P}$ for every $|\eps|$ sufficiently small.
\end{itemize}
Indeed, let $\phi \in \mathcal{C}^\infty_c(a,b) \setminus \{0\}$ be negative. Taking $q=\ddot{\phi}$, we obtain
\[
\int_a^b \ddot{\phi} \psi_0 = \int_a^b \phi \ddot{\psi}_0 = -\int_a^b g'\left( u(t;a,b;q) \right) \phi \psi_0 >0,
\]
because $-g'<0$ in $\R$ and $\psi_0>0$ in $(a,b)$. Also, it is easy to check that the function $p +\eps q \in \mathcal{P}$ whenever $|\eps|$ is sufficiently small (see Remark \ref{rem: su P}). So, by definition, $\mathcal{F}\left(u(\cdot;a,b;p+\eps q)\right) = p+\eps q$ (to ensure that $u(\cdot;a,b;p+\eps q)$ solves \eqref{positive BVP} with forcing term $p + \eps q$, it is essential to know that $p + \eps q \in \mathcal{P}$), and by Lemma \ref{lem: continuous dependence} it results $u(\cdot;a,b;p+\eps q) \to u(\cdot;a,b;p)$ in $X$ as $\eps \to 0^-$. On the other hand, by Theorem \ref{thm: on ordinary singular} there exists a neighbourhood $U$ of $u(\cdot;a,b;p)$ in $X$ such that the equation $\mathcal{F}(u)=p+\eps q$ has no solution in $U$ for $\eps<0$ sufficiently small, a contradiction.
\end{proof}

As an easy consequence of the Fredholm alternative, we obtain also the following corollary.

\begin{corollary}\label{cor: uniq non homogeneous}
Let $p$ satisfy \eqref{h2}, let $\mathcal{P}$ be defined by \eqref{def di P}, let $\tilde L$ be defined in Proposition \ref{prop:existence}, and assume that $b-a \ge \tilde L$. The boundary value problem
\[
\begin{cases}
\ddot{\psi}(t)+g'(u(t;a,b;q)) \psi(t) = 0 & t \in (a,b) \\
\psi(a)= \psi_a, \quad \psi(b)=\psi_b
\end{cases}
\]
has a unique solution for every $q \in \mathcal{P}$.
\end{corollary}

\section{Differentiability of $\varphi^+(a,b)$}\label{sec: diff}

In this section we will show that $\varphi^+(a,b) = J_{(a,b),p}(u(\cdot\,;a,b;p))$ is differentiable as function of $a$ and $b$.

\begin{lemma}\label{lem: diff of minimizer}
Let $p$ satisfy \eqref{h2}, and let $\mathcal{P}$ be defined by \eqref{def di P}. Let $A$ and $B$ be fixed and let
\[
 \mathcal I := \left\{(t,a,b)\in \R^3: b-a > \tilde L, \  A  < a \le t \le b < B\right\},
\]
where $\tilde L$ has been defined in Proposition \ref{prop:existence}. If $q \in \mathcal{P}$ is of class
$\mathcal{C}^1$, then the map
\[
(t,a,b) \in \mathcal{I} \mapsto \left( u(t;a,b;q), \dot{u}(t;a,b;q) \right) \in \R^2
\]
is of class $\mathcal{C}^1$, too. More precisely,
\begin{align*}
\frac{\pa u}{\pa a}(t;a,b;q) = \xi_1(t) \quad & \quad \frac{\pa \dot{u}}{\pa a}(t;a,b;q) = \dot{\xi}_1(t) \\
\frac{\pa u}{\pa b}(t;a,b;q) = \xi_2(t) \quad & \quad \frac{\pa \dot{u}}{\pa b}(t;a,b;q) = \dot{\xi}_2(t),
\end{align*}
where $\xi_1$ and $\xi_2$ are the solutions (unique by Corollary \ref{cor: uniq non homogeneous}) of
\[
\ddot{\xi}(t) + g'(u(t;a,b;q)) \xi(t) = 0
\]
with the boundary conditions
\[
\begin{cases}
\xi_1(a)= - \dot{u}(a^+;a,b;q) \\
\xi_1(b) = 0
\end{cases} \quad \text{or} \quad \begin{cases}
\xi_2(a)=0 \\
\xi_2(b)= - \dot{u}(b^-;a,b;q),
\end{cases}
\]
respectively.
\end{lemma}

\begin{proof}
In light of the results of the previous sections, it is not difficult to adapt the proof of Lemma 5.1 in \cite{OrVe}. We report the sketch of the proof for the sake of completeness. Thanks to the explicit relations \eqref{scaling}-\eqref{scaled problem}, the first part of the thesis follows if we prove the differentiability of $\wh{u}(\cdot\,;a,b;q)$ with respect to $(t,a,b)$. Let $\Delta:= \{(a,b) \in \R^2: b-a > \tilde L, A < a < b < B\}$, $X=H_0^1(0,1) \cap H^2(0,1)$, and consider the map $\Phi: \Delta \times X \to L^2(0,1)$ defined by
\[
\Phi(w,a,b) = -\ddot{w} - g\left((b-a)^2w\right) + q(a+t(b-a)).
\]
By definition, $\Phi\left(\wh{u}(\cdot\,;a,b;q);a,b\right) = 0$; we wish to show that the implicit function theorem applies to $\Phi$ in a neighbourhood of $\wh{u}(\cdot\,;a,b;q)$. Having chosen $q \in \mathcal{C}^1(\R)$, it is not difficult to check that $\Phi \in \mathcal{C}^1(\Delta \times X, Y)$, and that in particular
\[
\pa_w \Phi\left( \wh{u}(\cdot\,;a,b;p);a,b\right)[\psi] = -\ddot{\psi} - (b-a)^2 g'\left( (b-a)^2 \wh{u}(\cdot\,;a,b;p)\right) \psi,
\]
which is invertible thanks to Proposition \ref{prop: non degeneracy}. Therefore, the implicit function theorem applies and the map $(a,b) \mapsto \wh{u}(\cdot\,;a,b;q)$ is of class $\mathcal{C}^1(\Delta,X)$. By looking at the topology of $X$, this means that the map
\[
(t,a,b) \in \mathcal{I} \mapsto \left( u(t;a,b;q), \dot{u}(t;a,b;q) \right) \in \R^2
\]
has partial derivatives with respect to $a$ and $b$, and that they are continuous in the three variables. The differential equation for $\wh{u}(\cdot\,;a,b;q)$ reveals that also the partial derivative with respect to $t$ exists and is continuous, that is, the map is $\mathcal{C}^1$, which completes the first part of the proof. At this point, the characterizations of $\xi_1$ and $\xi_2$ can be obtained by differentiating problem \eqref{positive BVP}  by $a$ and $b$ respectively.
\end{proof}
\begin{proposition}\label{prop: diff of phi}
For every $p$ satisfying \eqref{h2}, the function \\$\varphi^+(a,b)=\varphi^+(a,b;p)$ is of class $\mathcal{C}^1$ with respect to $a$ and $b$ in $\{ b-a > \tilde L\}$, with derivatives
\[
\frac{\pa \varphi^+}{\pa a}(a,b) = \frac{1}{2} \dot{u}^2(a^+;a,b;p) \quad \text{and} \quad \frac{\pa \varphi^+}{\pa b}(a,b) = -\frac{1}{2} \dot{u}^2(b^-;a,b;p).
\]
\end{proposition}
\begin{proof}
If $p \in \mathcal{C}^1(\R)$ then we can apply Lemma \ref{lem: diff of minimizer}, obtaining that $\varphi^+(a,b)= J_{(a,b),p}(u(\cdot\,;a,b;p))$ is differentiable. In such case, the expressions of its derivatives follow by direct computation. In the general case, we claim that
\begin{equation}\label{claim strano}
\text{there exists $(q_n) \subset \mathcal{P} \cap \mathcal{C}^1(\R)$ such that $q_n \to p$ in $L^2(A,B)$}.
\end{equation}
This is not straightforward, since $\mathcal{P}$ is defined as in \eqref{def di P}. Let $\eps_n \to 0$ as $n \to \infty$, an let us consider the decomposition
\[
p= p_{1,\eps_n} + \dot{p}_{2,\eps_n}
\]
given by Lemma \ref{lem: decomposition}. For any fixed $n$, we consider
\[
q_{n,m} = A(p) + \frac{d}{dt}\left(\rho_m \ast p_{2,\eps_n}\right) = A(p) + \rho_m \ast \dot{p}_{2,\eps_n},
\]
where $(\rho_m)$ is a family of mollifiers, $\ast$ denotes the usual product of convolution, and the last identity follows from the fact that $p_{2,\eps} \in \mathcal{C}^1(\R)$. It is not difficult to check that $q_{n,m} \in \mathcal{P} \cap \mathcal{C}^1(\R)$ for any $m,n$, and that for any $n$ there exists $m_n$ sufficiently large such that
\[
\|q_{n,m_n} - p \|_{L^2(A,B)} < \eps_n.
\]
Hence, the sequence $(q_{n,m_n})$ has the desired properties, and claim \eqref{claim strano} follows. \\
We introduce $\varphi_n(a,b):= \varphi^+(a,b;q_n)$ and $\varphi(a,b):= \varphi^+(a,b;p)$, and observe that, thanks to the previous step, each $\varphi_n$ is of class $\mathcal{C}^1(\R)$. Let $\Delta:= \{(a,b): b-a > \tilde L, A < a < b < B\}$. We claim that
\begin{equation}\label{claim 2}
\text{$\varphi_n\to \varphi$ uniformly for $(a,b) \in \overline{\Delta}$}.
\end{equation}
If not,
\[
\sup_{(a,b) \in \overline{\Delta}} \left| \varphi_n(a,b) - \varphi(a,b) \right| = \sup_{(a,b) \in \overline{\Delta}} \left| \varphi^+(a,b;q_n) - \varphi^+(a,b;p) \right|  = c_n \ge \bar c >0.
\]
By Lemma \ref{lem: continuous dependence} and the continuity of $J_{(a,b),p}(u)$ as function of $(u,a,b,p)$, the function $\varphi^+$ is continuous in the three variables, so that by compactness for every $n$ the supremum is achieved by $(a_n,b_n) \in \overline{\Delta}$. Therefore, if \eqref{claim 2} does not hold, then
\[
\left|\varphi^+(a_n,b_n;q_n)- \varphi^+(a_n,b_n;p) \right| \ge \bar c
\]
for any $n$. Since, up to subsequences, both $a_n$ and $b_n$ converge, this contradicts the continuity of $\varphi^+$.

With a similar argument we see also that $\dot{u}(\tau;a,b;q_n) \to \dot{u}(\tau;a,b;p)$ for $\tau = a,b$, uniformly in $\overline{\Delta}$, so that
\[
\frac{\pa \varphi_n}{\pa a}(a,b) \to \frac{1}{2}\dot{u}^2 (a^+;a,b;p) \quad \text{and} \quad \frac{\pa \varphi_n}{\pa b}(a,b) \to -\frac{1}{2}\dot{u}^2 (b^-;a,b;p),
\]
uniformly in $\overline{\Delta}$. The convergence of $(\varphi_n)$ and of the sequences of the derivatives reveals that $\varphi$ is of class $\mathcal{C}^1$ in $\Delta$, and the thesis follows.
\end{proof}

\section{Sign-changing solutions}\label{sec: sign-change}

In this section we complete the proof of Theorem \ref{thm: main thm}. Firstly, we prove the existence of sign-changing solutions of \eqref{equation} in bounded (sufficiently large) intervals; then, by an exhaustion  procedure, we pass to the whole real line. To do this, we juxtapose positive and negative solutions on adjacent intervals, the latter existing and satisfying analogous properties of the former ones, as enlightened in Remark \ref{rmk: sul pb negativo}.
To distinguish between positive and negative solutions, and since the forcing term $p$ is now fixed, we change our notations accordingly, denoting such solutions as $u_\pm(\cdot\,;a,b)$. Resuming, we have the following result.
\begin{proposition}\label{key proposition resumed}
For every $\eps>0$ there exists $L>0$ such that, if $b-a \ge L$, then the value $\varphi^\pm(a,b)$ is achieved by a unique $u_\pm(\cdot\,;a,b) \in H_0^1(a,b)$, which is strictly positive/negative and solves equation \eqref{equation} in $(a,b)$. Moreover,
\begin{gather*}
\|u_\pm(\cdot\,;a,b)\| \le (\|g\|_{\infty}+\|p\|_\infty) (b-a)^{\frac{3}{2}} \\
-\underline{\alpha}(b-a)^3   \le \varphi^+(a,b) \le -\overline{\alpha}(b-a)^3 \\
-\underline{\beta}(b-a)^3   \le \varphi^-(a,b) \le -\overline{\beta}(b-a)^3,
\end{gather*}
where $\underline{\alpha},\overline{\alpha}$ have been defined in \eqref{alfasottotsopra} and
\begin{equation*}
\underline{\beta} := \frac{(-g_- + A(p))^2}{24}+\eps \quad \text{and} \quad \overline{\beta}:= \frac{(-g_- + A(p))^2}{24}-\eps.
\end{equation*}
\end{proposition}
\begin{proof}
The proposition directly follows from Proposition \ref{prop:existence}, Lemma \ref{bound norm minimizers}, Corollary \ref{indietro su (a,b)} and Remark \ref{rmk: sul pb negativo}.
\end{proof}
By assumption \eqref{h2}, there are two possibilities:
\[
\text{either} \quad g_+-A(p) = -g_- +A(p) \quad \text{or} \quad g_+-A(p)  \neq -g_- +A(p).
\]
In the former case, we observe that for a given $\eps$ it results $\underline{\alpha}= \underline{\beta}$ and $\overline{\alpha}= \overline{\beta}$. Otherwise, it is possible to choose $\eps$ sufficiently small in such a way that
\[
\text{either} \quad \underline{\alpha} < \overline{\beta} \quad \text{or} \quad \underline{\beta}< \overline{\alpha}.
\]
To fix the ideas, in the following we consider the case
\begin{equation}\label{sort alhpa beta}
\overline{\beta}<\underline{\beta}<\overline{\alpha} < \underline{\alpha}.
\end{equation}
The reader can easily adapt the arguments below in order to cover also the other situations (actually, if $g_+-A(p) = -g_- +A(p)$, the problem is considerably simplified).

Firstly, we start by choosing $\eps>0$ sufficiently small in Proposition \ref{key proposition resumed} in such a way that
\begin{equation}\label{rel alpha beta}
\ddfrac{\underline{\beta}}{ \left(1+\sqrt{\underline{\beta}/\underline{\alpha}}\right)^2 } < \overline{\beta};
\end{equation}
by definition, one can easily check that this choice is possible.

\begin{remark}\label{rmk: su scelta alpha beta}
Let $\nu:= \underline{\beta}/ \underline{\alpha}$. It is useful to observe that equation \eqref{rel alpha beta} implies that
\begin{align*}
\underline{\alpha} \left(\frac{\sqrt{ \nu  }}{1+\sqrt{\nu}}\right)^3 + \underline{\beta}\left(\frac{1}{1+\sqrt{\nu}}\right)^3-\overline{\beta} & <0 \\
\underline{\alpha} \left(\frac{\sqrt{\nu}}{1+\sqrt{\nu}}\right)^3 + \underline{\beta}\left(\frac{1}{1+\sqrt{\nu}}\right)^3-\overline{\alpha} & <0.
\end{align*}
First of all, by \eqref{sort alhpa beta} we immediately see that the second of these relations is automatically satisfied provided the first one holds. And for the first one it is sufficient to note that
\begin{multline*}
\underline{\alpha} \left(\frac{\sqrt{ \nu  }}{1+\sqrt{\nu}}\right)^3 + \underline{\beta}\left(\frac{1}{1+\sqrt{\nu}}\right)^3  = \underline{\alpha} \left(\frac{1}{1+\sqrt{\nu}}\right)^3  \left[  \left(\sqrt{\nu} \right)^3 + \nu \right] \\
= \frac{\underline{\alpha} \nu}{\left(1+ \sqrt{\nu}\right)^2} =  \frac{\underline{\beta}}{ \left(1+\sqrt{\underline{\beta}/\underline{\alpha}}\right)^2 }.
\end{multline*}
\end{remark}

Let $(A,B) \subset \R$ and $k \in \N$ be such that $(k+1)L \le B-A$; hence, it is possible to divide the interval $(A,B)$ in $k+1$ sub-intervals, in such a way that each of them is larger than $L $. We define the set of  admissible partitions of $(A,B)$ in $(k+1)$ sub-intervals as
\[
\mathcal{B}_k:= \left\{(t_1,\dots,t_k) \in \R^k: A=: t_0 \le t_1 \le \dots \le t_k \le t_{k+1}:=B, t_{i+1}-t_i \ge L \right\};
\]
also, we introduce the function $\psi: \mathcal{B}_k \to \R$ defined by
\begin{equation}\label{def di psi}
\psi(t_1,\dots,t_k):= \sum_{i=0}^k \varphi^{\sigma(i)} (t_i,t_{i+1}), \quad \text{where} \quad \sigma(i) = \begin{cases} + & \text{if $i$ is even} \\ - & \text{if $i$ is odd}. \end{cases}
\end{equation}
We consider the maximization problem
\begin{equation}\label{max problem}
c_k(A,B):= \sup \left\{ \psi(t_1,\dots,t_k): (t_1,\dots,t_k) \in \mathcal{B}_k \right\}.
\end{equation}

\begin{remark}
It is possible to consider also the maximization problem for the function having opposite $\sigma(i)$. The situation is essentially the same.
\end{remark}

\begin{lemma}\label{lem: exist of max}
The value $c_k(A,B)$ is achieved by a partition $(\bar t_1,\dots,\bar t_k) \in \mathcal{B}_k$.
\end{lemma}
\begin{proof}
This follows from the continuity of $\varphi^{\sigma(i)}$ (in fact $\varphi^{\sigma(i)}$ is differentiable, Proposition \ref{prop: diff of phi}), and from the compactness of $\mathcal{B}_k$.
\end{proof}

To each interval $(\bar t_i, \bar t_{i+1})$ we associate
\[
u_i:= u_{\sigma(i)}(\cdot\,; \bar t_i, \bar t_{i+1}).
\]
In this way, it is defined on the whole $[A,B]$ a function
\begin{equation}\label{sol incollata}
u_{(A,B),k}(t) :=  u_i(t) \quad \text{if $t \in [\bar{t}_i,\bar{t}_{i+1}]$},
\end{equation}
which is a solution of \eqref{equation} in $(A,B) \setminus \{\bar t_1,\dots,\bar t_k\}$, and has exactly $k$ zeros in $(A,B)$. If we show that it is differentiable in each $\bar t_i$, then $u_{(A,B),k}$ will be a solution in the whole $(A,B)$. To prove the smoothness of $u_{(A,B),k}$, we wish to exploit the knowledge of the explicit expression of the derivatives of $\varphi^{\sigma(i)}$, given in Proposition \ref{prop: diff of phi}. Having this in mind, we observe that, if $(\bar t_1,\dots, \bar t_k)$ is an inner point of $\mathcal{B}_k$, then by maximality it results $\nabla \psi (\bar t_1,\dots,\bar t_k)=0$, where the partial derivatives of $\psi$ can be expressed in terms of the partial derivatives of $\varphi^{\sigma(i)}$. Therefore, the next step consists in the proof of the following lemma.
\begin{lemma}\label{prop: max interno}
There exists $H$, depending only on $L$ and on $p$, such that for any $(A,B)\subset\R$, $k \in \N$ with
\[
B-A \ge H(k+1),
\]
the corresponding maximizing partition $(\bar t_1,\dots,\bar t_k) \in \mathcal{B}_k$ is an inner point of $\mathcal{B}_k$, that is, $\bar t_{i+1}-\bar t_i >L$ for every $i$.
\end{lemma}
We need two intermediate results. The first one says that the ratio between two adjacent sub-intervals of a maximizing partition can be controlled by means of a positive constant depending only on $L$ and on $p$.

\begin{lemma}\label{lem: 6.1}
Let $(\bar t_1,\dots,\bar t_k) \in \mathcal{B}_k$ be a maximizing partition for \eqref{max problem}. There exists $\bar h \ge 1$, depending only on $L$ and on $p$, such that
\[
\frac{1}{\bar h}(\bar t_{i}- \bar t_{i-1}) \le \bar t_{i+1}-\bar t_i \le \bar h (\bar t_i-\bar t_{i+1})
\]
for every $i=1,\dots,k$.
\end{lemma}

\begin{proof}
For an arbitrary $i$, let $\lambda= \bar{t}_i - \bar t_{i-1}$ and $h\lambda= \bar t_{i+1}-\bar t_i$. We wish to show that $h$ is bounded from below and from above by two positive constants depending only on $L$ and on $p$. Let $\nu:=  \underline{\beta}/ \underline{\alpha}$, which belongs to $(0,1)$ by \eqref{sort alhpa beta}. If both $\lambda$ and $h\lambda$ are smaller than or equal to $L/\sqrt{\nu}$, then $\sqrt{\nu} \le h \le 1/\sqrt{\nu}$. Otherwise, at least one between $\lambda$ and $h\lambda$ is greater then $ L/\sqrt{\nu}$, so that
\begin{equation}\label{maggiorazione somma 2 int}
(1+h) \lambda > \left(1+ \frac{1}{\sqrt{\nu}} \right) L.
\end{equation}
Firstly, let us consider the case $\sigma(i-1)=+$, that is, $i-1$ is even. Let
\[
s:= \bar t_{i-1} + \frac{\sqrt{\nu}}{1+\sqrt{\nu}} (\bar t_{i+1}- \bar t_{i-1}) \in (\bar t_{i-1},\bar t_{i+1}).
\]
We consider the variation of $(\bar t_1,\dots,\bar t_k)$ obtained replacing $\bar t_i$ with $s$. This is an admissible partition in $\mathcal{B}_k$, as by \eqref{maggiorazione somma 2 int} we have
\begin{align*}
s-\bar t_{i-1} &= \frac{\sqrt{\nu}}{1+\sqrt{\nu}}(1+h)\lambda > \frac{\sqrt{\nu}}{1+\sqrt{\nu}} \left(1+ \frac{1}{\sqrt{\nu}}\right) L = L \\
\bar t_{i+1}-s &= \frac{1}{1+\sqrt{\nu}}(1+h)\lambda > \frac{1}{1+\sqrt{\nu}} \left(1+ \frac{1}{\sqrt{\nu}}\right) L > L.
\end{align*}
The variational characterization of $(\bar t_1,\dots, \bar t_k)$ implies that
\[
\psi(\bar t_1,\dots, \bar t_{i-1}, s , \bar t_{i+1},\dots, \bar t_k) \le \psi (\bar t_1,\dots, \bar t_{i-1}, \bar t_i , \bar t_{i+1},\dots, \bar t_k);
\]
by definition, this means
\[
\varphi^{\sigma(i-1)}(\bar t_{i-1},s) + \varphi^{\sigma(i)}(s,\bar t_{i+1}) \le \varphi^{\sigma(i-1)}(\bar t_{i-1},\bar t_i) + \varphi^{\sigma(i)}(\bar t_{i},\bar t_{i+1}).
\]
Therefore, recalling that we are considering the case $\sigma(i-1)=+$, by Proposition \ref{key proposition resumed} we deduce
\[
-\underline{\alpha} \left(\frac{\sqrt{\nu}}{1+\sqrt{\nu}}\right)^3 (1+h)^3 \lambda^3 - \underline{\beta} \left(\frac{1}{1+\sqrt{\nu}}\right)^3 (1+h)^3 \lambda^3 \le -\overline{\alpha}\lambda^3 -\overline{\beta} h^3\lambda^3,
\]
that is,
\begin{multline*}
\left[ \underline{\alpha} \left(\frac{\sqrt{\nu}}{1+\sqrt{\nu}}\right)^3 + \underline{\beta}\left(\frac{1}{1+\sqrt{\nu}}\right)^3-\overline{\beta}\right] h^3 \\
+ 3\left[ \underline{\alpha} \left(\frac{\sqrt{\nu}}{1+\sqrt{\nu}}\right)^3 + \underline{\beta}\left(\frac{1}{1+\sqrt{\nu}}\right)^3 \right] (h^2+h) \\
\left[ \underline{\alpha} \left(\frac{\sqrt{\nu}}{1+\sqrt{\nu}}\right)^3 + \underline{\beta}\left(\frac{1}{1+\sqrt{\nu}}\right)^3-\overline{\alpha}\right] \ge 0.
\end{multline*}
As observed in Remark \ref{rmk: su scelta alpha beta}, thanks to the choice \eqref{rel alpha beta}, the coefficient of $h^3$ and the last term are negative, so that this relation cannot be satisfied if $h$ is too small or too large: this implies that necessarily $1/\bar h_1 \le h \le \bar h_1$ for a positive constant $\bar h_1>1$, which depends only on $L$ and on $p$.

In case $\sigma(i-1)=-$, one can follow the same line of reasoning, replacing the previous definition of $s$ with
\[
s:= \bar t_{i-1} + \frac{1}{1+\sqrt{\nu}} (\bar t_{i+1}- \bar t_{i-1}) \in (\bar t_{i-1},\bar t_{i+1}).
\]
Again, the relation
\[
\psi(\bar t_1,\dots, \bar t_{i-1}, s , \bar t_{i+1},\dots, \bar t_k) \le \psi (\bar t_1,\dots, \bar t_{i-1}, \bar t_i , \bar t_{i+1},\dots, \bar t_k)
\]
implies that for the quantity $\bar h_1 >1$ previously introduced it results $1/\bar h_1 \le h \le \bar h_1$, and the desired result follows choosing $\bar h :=\max\{1/\sqrt{\nu},\bar h_1\}$.
\end{proof}
Now we can show that, in a maximizing partition, the ratio between the larger sub-interval and the smaller one is bounded by a constant depending only on $L$ and on $p$.
\begin{lemma}\label{lem: 6.2}
Let
\[
\underline{\lambda}:= \min_{i} \left(\bar t_{i+1} - \bar t_i \right) \quad \text{and} \quad \overline{\lambda}:= \max_i \left(\bar t_{i+1} - \bar t_i \right).
\]
Then there exists $h^* \ge 1$, depending only on $L$ and on $p$, such that
\[
\overline{\lambda} \le h^* \underline{\lambda}.
\]
\end{lemma}

\begin{proof}
Let us denote with $i \neq j$, $0 \le i,j\le k$,  two indexes such that
\[
\overline{\lambda}=\bar t_{i+1} - \bar t_i \quad \text{and} \quad \underline{\lambda}= \bar t_{j+1}-\bar t_j.
\]
To fix the ideas we consider the case $i<j$. As the previous lemma asserts that the length of any interval is comparable with the one of its neighbors, we can assume without loss of generality $i$ and $j$ to be even, $k\geq5$ and $j-i \ge 4$, i.e. $i+2 \le j-2$. Let us set again $\nu :=\underline{\beta}/\underline{\alpha}$, and let
\[
\bar \sigma:= \frac{1}{2}\left( \frac{1}{2}+\frac{1}{\sqrt[3]{2}} \right) \frac{\sqrt{\nu}}{1+\sqrt{\nu}}.
\]
If $\overline{\lambda} \le \max \{ L /\bar \sigma , L/(1-2\bar \sigma) \}$, we can choose $h^*= \max \{ 1/\bar \sigma, 1/(1-2\bar \sigma) \}$. Otherwise, we consider a variation of $(\bar t_1,\dots,\bar t_k)$ introducing two points
\[
s_1:= \bar t_i + \bar \sigma (\bar t_{i+1}-\bar t_i) \quad \text{and} \quad s_2:= \bar t_i + (1-\bar \sigma) (\bar t_{i+1}-\bar t_i).
\]
between $\bar t_i$ and $\bar t_{i+1}$, and eliminating $\bar t_j$ and $\bar t_{j+1}$ if $j<k$; if $j=k$, we eliminate $\bar t_{k-1}$ and $\bar t_k$. For the reader's convenience, we explicitly observe that, since $\nu \in (0,1)$, it results $\bar t_i <s_1<s_2<\bar t_{i+1}$.
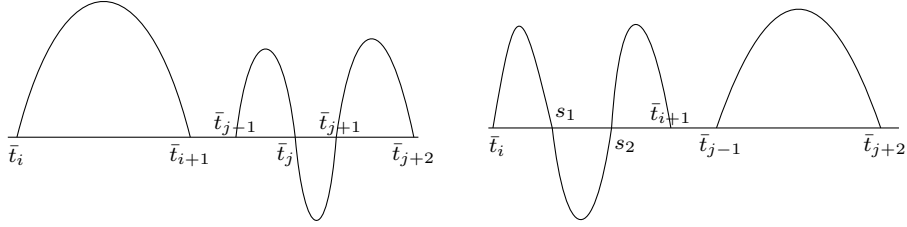
\begin{center}
\begin{tikzpicture}[scale=0.6]
\draw(-5,-1)--(4,-1);
\draw (-4.8,-1)  .. controls (-3.8,3) and (-2,3)..(-1,-1);
\draw (0,-1) .. controls (0.3,1.6) and (1,1.6).. (1.3,-1);
\draw (1.3,-1) .. controls (1.5,-3.3) and (2,-3.6).. (2.2,-1);
\draw (2.2,-1) .. controls (2.6,1.9) and (3.3,1.9).. (3.9,-1);
\draw (-4.8,-1.4)[font=\footnotesize] node{$\bar t_i$};
\draw (-1,-1.4)[font=\footnotesize] node{$\bar t_{i+1}$};
\draw (0,-0.7)[font=\footnotesize] node{$\bar t_{j-1}$};
\draw (1.1,-1.4)[font=\footnotesize] node{$\bar t_{j}$};
\draw (2.3,-0.7)[font=\footnotesize] node{$\bar t_{j+1}$};
\draw (3.9,-1.4)[font=\footnotesize] node{$\bar t_{j+2}$};
\end{tikzpicture}
\quad \begin{tikzpicture}[scale=0.6]
\draw(-5,-1)--(4,-1);
\draw (-4.9,-1)  .. controls (-4.4,2) and (-4.3,2)..(-3.6,-1);
\draw (-3.6,-1)  .. controls (-3.3,-3.6) and (-2.7,-3.8)..(-2.3,-1);
\draw (-2.3,-1)  .. controls (-2.1,2.1) and (-1.5,2)..(-1,-1);
\draw (0,-1) .. controls (1.2,2.5) and (2.4,2.5).. (3.6,-1);
\draw (-4.8,-1.4)[font=\footnotesize] node{$\bar t_i$};
\draw (-1,-0.7)[font=\footnotesize] node{$\bar t_{i+1}$};
\draw (0.1,-1.4)[font=\footnotesize] node{$\bar t_{j-1}$};
\draw (3.7,-1.4)[font=\footnotesize] node{$\bar t_{j+2}$};
\draw (-3.3,-0.7)[font=\footnotesize] node{$s_1$};
\draw (-2,-1.4)[font=\footnotesize] node{$s_2$};
\end{tikzpicture}
\end{center}
In what follows, the notation corresponds to the case $j<k$. \\
As $\overline{\lambda}>\max \{ L /\bar \sigma , L/(1-2\bar \sigma) \}$, the new partition is in $\mathcal{B}_k$: indeed
\begin{equation}\label{exp intervallini}
\begin{split}
s_1- \bar t_i = \bar \sigma  (\bar t_{i+1}-\bar t_i) = \bar \sigma \overline{\lambda}  &> L \\
s_2-s_1  = (1-2 \bar \sigma) (\bar t_{i+1}-\bar t_i) = (1-2\bar \sigma) \overline{\lambda} &> L \\
\bar t_{i+1}-s_2 = \bar \sigma  (\bar t_{i+1}-\bar t_i) = \bar \sigma \overline{\lambda} &> L.
\end{split}
\end{equation}
As a consequence, by maximality,
\[
\psi(\bar t_1, \dots,\bar t_{i},s_1,s_2,\bar t_{i+1}, \dots, \bar t_{j-1}, \bar t_{j+2}, \dots, \bar t_k)
\le \psi (\bar t_1,\dots,\bar t_k),
\]
that is,
\begin{multline*}
\varphi^+(\bar t_i,s_1) + \varphi^-(s_1,s_2) + \varphi^+(s_2,\bar t_{i+1}) + \varphi^+(\bar t_{j-1}, \bar t_{j+2}) \\
\le \varphi^+(\bar t_i, \bar t_{i+1}) + \varphi^+(\bar t_{j-1},\bar t_j) + \varphi^-(\bar t_j,\bar t_{j+1}) + \varphi^+(\bar t_{j+1}, \bar t_{j+2}).
\end{multline*}
We know that $\varphi^+(\bar t_i, \bar t_{i+1}) \le -\overline{\alpha} \overline{\lambda}^3$, and the other terms on the right hand side are negative; on the other hand, for the left hand side we can use the expressions \eqref{exp intervallini} and the fact that, by Lemma \ref{lem: 6.1}, $\bar t_{j+2}- \bar t_{j-1}  \le (2 \bar h +1 ) \underline{\lambda}$. Therefore
\[
- 2 \underline{\alpha} \bar{\sigma}^3 \overline{\lambda}^3 - \underline{\beta} (1-2 \bar \sigma)^3 \overline{\lambda}^3- \underline{\alpha} (2\bar h +1)^3 \underline{\lambda}^3 \le -\overline{\alpha} \overline{\lambda}^3,
\]
which gives
\[
\left[\overline{\alpha} - 2\underline{\alpha} \bar \sigma^3  - \underline{\beta}(1-2\bar \sigma)^3    \right]\left(\frac{\overline{\lambda}}{\underline{\lambda}}\right)^3 \le \underline{\alpha}(2 \bar h +1)^3.
\]
We claim that
\begin{equation*}
\overline{\alpha} - 2\underline{\alpha} \bar \sigma^3  - \underline{\beta}(1-2\bar \sigma)^3  > 0;
\end{equation*}
as a consequence, the thesis will follows. To show the claim, we note that, by definition of $\bar \sigma$, it results
\[
2 \bar \sigma^3 < \left(\frac{\sqrt{\nu}}{1+\sqrt{\nu}} \right)^3 \quad \text{and} \quad (1-2\bar \sigma)^3 < \left(\frac{1}{1+\sqrt{\nu}} \right)^3;
\]
Thanks to the choice \eqref{rel alpha beta}, recalling also Remark \ref{rmk: su scelta alpha beta}, we easily deduce
\[
\overline{\alpha} - 2\underline{\alpha} \bar \sigma^3  - \underline{\beta}(1-2\bar \sigma)^3 > \overline{\alpha} - \underline{\alpha} \left(\frac{\sqrt{\nu}}{1+\sqrt{\nu}}\right)^3 - \underline{\beta}\left(\frac{1}{1+\sqrt{\nu}}\right)^3 >0,
\]
which completes the proof.
\end{proof}
\begin{proof}[End of the proof of Lemma \ref{prop: max interno}]
Let $H= h^*(L+1)$, with $h^*$ introduced in Lemma \ref{lem: 6.2}. Then any partition of an interval of length $B-A \ge H(k+1)$ in $k+1$ sub-intervals has a sub-interval larger than $h^*(L+1)$, and in particular $\overline{\lambda} \ge h^*(L+1)$. Applying Lemma \ref{lem: 6.2}, we immediately deduce $\underline{\lambda} \ge  L+1$.
\end{proof}

We are ready to prove the existence of sign-changing solutions of \eqref{equation} in large intervals.

\begin{proposition}\label{prop: sign-changing in bounded}
There exists $H$, depending only on $L$ and on $p$, such that if $B-A \ge H(k+1)$ and $(\bar t_1,\dots, \bar t_k)$ is a maximizing partition for \eqref{max problem}, then the function $u_{(A,B),k}$ defined by \eqref{sol incollata} is a solution of \eqref{equation}.
\end{proposition}

\begin{proof}
By construction, $u_{(A,B),k}$ solves \eqref{equation} in $(A,B) \setminus \{\bar t_1,\dots, \bar t_k\}$. Moreover, by Lemma \ref{prop: max interno}, $(\bar t_1,\dots, \bar t_k)$ is a free critical point of the function $\psi$, so that $\nabla \psi (\bar t_1,\dots, \bar t_k)=0$. In view of Proposition \ref{prop: diff of phi}, this writes
\[
-\frac{1}{2} \dot{u}_{i-1}^2(\bar t_{i}^-) +\frac{1}{2} \dot{u}_{i}^2(\bar t_{i}^+) = 0 \qquad i=1,\dots,k.
\]
But then  $u_{(A,B),k}$ is $\mathcal{C}^1$ across each $\bar t_{i}$, and the proposition follows.
\end{proof}

\begin{remark}\label{rem: stime finali}
Directly from the construction of $u_{(A,B),k}$, it is possible to obtain some estimates which will be useful in the next proof; we keep here the notation previously introduced. First of all, we note that for every $t \in (A,B)$ there exists $i$ such that $t \in [\bar t_i, \bar t_{i+1})$. Thanks to Lemma \ref{bound norm minimizers}, we deduce that
\begin{align*}
|u_{(A,B),k}(t)| & = |u_i(t)|  \le C (\bar t_{i+1}- \bar t_i)^2 \le C \overline{\lambda}^2 \\
|\dot{u}_{(A,B),k}(t)| & = |\dot{u}_i(t)| \le C (\bar t_{i+1}- \bar t_i) \le  C \overline{\lambda},
\end{align*}
where $C$ is a positive constant depending only on $g$ and $p$. As a consequence
\[
\|u_{(A,B),k}\|_{L^\infty(A,B)} \le C \overline{\lambda}^2 \quad \text{and} \quad \|\dot{u}_{(A,B),k}\|_{L^\infty(A,B)}  \le C \overline{\lambda}.
\]
On the other hand, let $\tau$ be a point of maximum of $|u_{(A,B),k}|$. There exists $j \in \{0,\dots,k\}$ such that $\tau \in (\bar t_j,\bar t_{j+1})$, so that by Corollary \ref{lem:stime su norme} it results
\[
\|u_{(A,B),k}\|_{L^\infty(A,B)} = | u_j(\tau)| \ge C_1(\bar t_{j+1}-\bar t_j) \ge C_1 \underline{\lambda},
\]
where $C_1$ is a positive constant depending only on $g$ and $p$.
\end{remark}

It is now possible to complete the proof of the main result.

\begin{proof}[Proof of Theorem \ref{thm: main thm}] For a fixed $L>\bar L$, let $\bar h$, $h^*$ and $H$ be as in Lemmas \ref{lem: 6.1}, \ref{lem: 6.2} and Proposition \ref{prop: sign-changing in bounded} respectively. Let $\mu \ge H$ be fixed (we explicitly remark that $h^*$ is independent of $\mu$). For every $n \in \N$ we have $2n\mu \ge 2n H$, so that by Proposition \ref{prop: sign-changing in bounded} there exists $u_{\mu,n}:= u_{(-\mu n,\mu n),2n-1}$ which is a solution of \eqref{equation} in $(-\mu n,\mu n)$ with $2n-1$ zeros, and its zeros correspond to a partition
\[
-\mu n=: \bar t_0 < \bar t_1 < \dots, \bar t_{2n-1}<\bar t_{2n} := \mu n,
\]
maximizing for $c_{2n-1}(-\mu n,\mu n)$, defined by \eqref{max problem}. At least one of the sub-intervals of the partition has to be smaller than or equal to $\mu$; recalling that $\underline{\lambda}:= \min_i (\bar t_{i+1} -\bar t_i)$ and $\overline{\lambda} = \max_i (\bar t_{i+1}-\bar t_i)$, it results $\underline{\lambda} \le \mu$; this implies, by means of Lemma \ref{lem: 6.2}, that $\overline{\lambda} \le h^* \mu$, where $h^*$ does not depend on $n$ or on $\mu$. Analogously, from the fact that at least one of the sub-intervals of the partition has to be larger than or equal to $\mu$, it is possible to deduce that $\underline{\lambda} \ge \mu/h^*$.

By using the estimates of Remark \ref{rem: stime finali}, it is immediate to obtain
\[
C_1 \left( \frac{\mu}{h^*} \right)^2 \le \| u_{\mu,n}\|_{L^\infty(-\mu n,\mu n)} \le C (h^* \mu)^2 \quad \text{and} \quad \| \dot u_{\mu,n}\|_{L^\infty(-\mu n,\mu n)} \le C (h^* \mu).
\]
Furthermore, being $u_{\mu,n}$ a solution of \eqref{equation}, it results
\[
\|\ddot{u}_{\mu,n}\|_{L^\infty(-\mu n,\mu n)} \le \|g\|_\infty+\|p\|_\infty.
\]
The previous estimates reveals that the sequence $(u_{\mu,n})_{n \in \N}$ is uniformly bounded in $W^{2,\infty}_{\loc}(\R)$, so that by the Ascoli-Arzel\`a theorem it converges in $\mathcal{C}^1_{\loc}(\R)$, up to a subsequence, to a function $u_{\mu}$ which is a solution of \eqref{equation} in the whole $\R$, and satisfies
\begin{equation}\label{uniform norm of bounded sol}
C_1 \left( \frac{\mu}{h^*} \right)^2 \le \| u_{\mu}\|_{L^\infty(\R)} \le C (h^* \mu)^2 \quad \text{and} \quad \| \dot u_{\mu}\|_{L^\infty(\R)} \le C (h^* \mu)
\end{equation}
By construction, $u_\mu$ has infinitely many zeros tending to infinity in both the directions; indeed, if this were not true, then $|u_\mu(t)| \ge C>0$ on an interval of length greater than $h^* \mu$, and by the $\mathcal{C}^1_{\loc}$ convergence the same should hold also for $u_{\mu,n}$ when $n$ is sufficiently large, which is not possible.

We have constructed a solution of \eqref{equation} defined in $\R$, which is bounded together with its first derivative. Now, we can obtain the sequence of bounded solutions $u_m=u_{\mu_m}$ simply repeating the same procedure for a sequence of parameters $\mu_m$ such that $\mu_m \to +\infty$ and \[
\mu_m > \sqrt{\frac{C}{C_1}} (h^*)^2 \mu_{m-1}
\]
for every $m$. Indeed, thanks to equation \eqref{uniform norm of bounded sol}, we deduce
\[
\| u_{m-1}\|_{L^\infty(\R)} \le C (h^* \mu_{m-1})^2 < C_1 \left(\frac{\mu_m}{h^*}\right)^2 \le \| u_{m}\|_{L^\infty(\R)},
\]
so that $u_{m-1} \not \equiv u_m$ and $\|u_m\|_\infty \to +\infty$ as $m \to \infty$.
\end{proof}

To conclude, as we mentioned in the introduction, we turn to the periodic framework. We keep the previous
notations, in particular $H$ is defined as in Lemma \ref{prop: max interno}. We have the following.
\begin{theorem}\label{thm: esistenza sub}
Let $g$ satisfy \eqref{h1}, and let $p$ be a continuous $T$-periodic function such that
\[
g_-<A(p) = \frac{1}{T} \int_0^T p(t)\,dt<g_+.
\]
Then, for any $(k,n) \in \N^2$ with $k$ odd and $n T \ge H(k+1)$, there exist a $nT$-periodic solution of \eqref{equation}, having exactly $k$ zeros in each interval of periodicity.
\end{theorem}
\begin{remark}
The nodal characterization of the solutions ensures that, whenever $T$ is
the minimal period of $p$, and  $n$ and $k+1$ are coprime integers, then $nT$ is the minimal period of the corresponding solution. This ensures the existence of an infinite sequence of subharmonic solutions, with diverging minimal period.
\end{remark}
\begin{proof}
Let
\[
\mathcal{A}_k:=\left\{(t_0,t_1,\dots,t_k) \in \R^k \left| \begin{array}{l} t_0 \le t_1 \le \dots \le t_k \le t_{k+1}:=t_0+nT,\\
t_{i+1}-t_i \ge L, \ t_0 \in [-T,2T] \end{array} \right.\right\},
\]
and let $\psi: \mathcal{A}_k \to \R$ defined as in \eqref{def di psi} (we point out that now $t_0$ is not fixed). There exists a maximizer $(\bar t_0,\bar t_1\dots, \bar t_k)$ for $\psi$. Since $p$ is $T$-periodic, we can assume $\bar t_0 \in [0,T)$. As a consequence, it results $\nabla \psi (\bar t_0,\bar t_1,\dots, \bar t_k) = 0$. The expression of the partial derivatives of $\psi$ with respect to $t_i$, $i=1,\dots,k$, says that the function $u_{(\bar t_0,\bar t_0+nT),k}$ (defined as in \eqref{sol incollata}) is a solution of \eqref{equation} in $(\bar t_0, \bar t_0+nT)$; also, the fact that $\pa_{t_0} \psi (\bar t_0,\bar t_1,\dots, \bar t_k)=0$ implies that
\[
- \frac{1}{2} \dot{u}_{(\bar t_0,\bar t_0+nT),k}^2 (\bar t_0^+) + \frac{1}{2} \dot{u}_{(\bar t_0,\bar t_0+nT),k}^2((\bar t_0+nT)^-) = 0,
\]
that is, $u_{(\bar t_0,\bar t_0+nT),k}$ can be extended by $nT$-periodicity as a (smooth) solution of \eqref{equation} in the whole $\R$.
\end{proof}

\subsection*{Acknowledgments}
We would like to thank professor Rafael Ortega for having suggested the problem and for all the kind and fruitful
discussions. Work partially supported by the PRIN2009 grant ``Critical Point Theory and Perturbative Methods for Nonlinear Differential Equations''.


\noindent\verb"n.soave@campus.unimib.it"\\
Dipartimento di Matematica e Applicazioni, Universit\`a degli Studi
di Milano-Bicocca, via Bicocca degli Arcimboldi 8, 20126 Milano,
Italy\bigskip

\noindent \verb"gianmaria.verzini@polimi.it"\\
Dipartimento di Matematica, Politecnico di Milano, p.za Leonardo da
Vinci 32,  20133 Milano, Italy

\end{document}